\documentclass[a4paper,12pt,reqno]{amsart}
\usepackage[T1]{fontenc}
\usepackage[utf8]{inputenc}
\usepackage{bera}
\usepackage{amssymb,dsfont,mathrsfs}
\usepackage[margin=1in]{geometry}
\usepackage{enumitem}
\usepackage{xparse}
\usepackage{graphicx}
\usepackage[bookmarksdepth=2]{hyperref}

\numberwithin{equation}{section}

\newtheorem{theorem}{Theorem}[section]
\newtheorem{lemma}[theorem]{Lemma}

\theoremstyle{definition}

\newtheorem{remark}[theorem]{Remark}

\renewcommand{\Im}{\operatorname{Im}}

\DeclareMathOperator{\Arg}{Arg}

\DeclareMathOperator{\sign}{sign}

\DeclareMathOperator{\ind}{\mathds{1}}

\newcommand{\sub}{\subseteq}
\newcommand{\pr}{\mathds{P}}
\newcommand{\ex}{\mathds{E}}
\newcommand{\C}{\mathds{C}}
\newcommand{\R}{\mathds{R}}
\newcommand{\Z}{\mathds{Z}}
\newcommand{\N}{\mathds{N}}
\newcommand{\ph}{\varphi}

\renewcommand{\le}{\leqslant}
\renewcommand{\ge}{\geqslant}

\NewDocumentCommand{\formula}{ssom}{%
 \IfBooleanTF{#1}{%
  \IfBooleanTF{#2}{%
   \IfValueTF{#3}%
    {\begin{align}\label{#3}\begin{gathered}#4\end{gathered}\end{align}}%
    {\begin{gather}#4\end{gather}}%
  }{%
   \IfValueTF{#3}%
    {\begin{align}\label{#3}\begin{aligned}#4\end{aligned}\end{align}}%
    {\begin{gather*}#4\end{gather*}}%
  }%
 }{%
  \IfValueTF{#3}%
   {\begin{align}\label{#3}#4\end{align}}%
   {\begin{align*}#4\end{align*}}%
 }%
}

\newcommand{\ignore}[1]{}

\begin{document}

\title[Bell-shaped sequences]{Bell-shaped sequences}
\author{Mateusz Kwaśnicki, Jacek Wszoła}
\thanks{Work supported by the Polish National Science Centre (NCN) grant no.\@ 2019/33/B/ST1/03098}
\address{Mateusz Kwaśnicki, Jacek Wszoła \\ Department of Pure Mathematics \\ Wrocław University of Science and Technology \\ ul. Wybrzeże Wyspiańskiego 27 \\ 50-370 Wrocław, Poland}
\email{mateusz.kwasnicki@pwr.edu.pl, 255718@student.pwr.edu.pl}
\date{\today}

\begin{abstract}
A nonnegative real function $f$ is said to be bell-shaped if it converges to zero at $\pm\infty$ and the $n$th derivative of $f$ changes sign $n$ times for every $n = 0, 1, 2, \ldots$\, In a similar way, we may say that a nonnegative sequence $a_k$ is bell-shaped if it converges to zero and the $n$th iterated difference of $a_k$ changes sign $n$ times for every $n = 0, 1, 2, \ldots$\, Bell-shaped functions were recently characterised by Thomas Simon and the first author. In the present paper we provide an analogous description of one-sided bell-shaped sequences. More precisely, we identify one-sided bell-shaped sequences with convolutions of Pólya frequency sequences and completely monotone sequences, and we characterise the corresponding generating functions as exponentials of appropriate Pick functions.
\end{abstract}

\maketitle

%
%

\section{Introduction}
\label{sec:intro}

A nonnegative real function $f$ is said to be \emph{bell-shaped} if it is smooth, it converges to zero at $\pm\infty$, and for every $n = 0, 1, 2, \ldots$ the $n$th derivative $f^{(n)}$ changes sign exactly $n$ times. According to~\cite{karlin}, this notion of bell-shaped functions was introduced in~1940s in the study of statistical games. Examples of bell-shaped functions include the densities of the normal distribution $\pi^{-1/2} \exp(-x^2)$, the Cauchy distribution $\pi^{-1} (1 + x^2)^{-1}$, and, more generally, stable distributions. The last claim become a popular conjecture after an incorrect proof appeared in~1983 in~\cite{gawronski}, and it was eventually proved in~\cite{kwasnicki}, extending a partial result due to Simon in~\cite{simon}. There are no compactly supported bell-shaped functions: this conjecture due to Schoenberg was proved already in~1950 by Hirschman in~\cite{hirschman}. However, many \emph{one-sided} functions, that is, functions supported in a half-line, are bell-shaped. Examples include the density functions of the Lévy distribution $\pi^{-1/2} x^{-3/2} e^{-1/x} \ind_{(0, \infty)}(x)$ and, more generally, of hitting times of 1-D diffusion processes; see~\cite{js}. The class of bell-shaped functions was completely characterised in~\cite{ks}, and it proved to be related to total positivity, infinite divisibility and the theory of Pick functions.

The concept of bell-shaped functions has its obvious discrete analogue: \emph{bell-shaped sequences}. A two-sided nonnegative sequence $(a_k)$ (with $k \in \Z$) is said to be \emph{bell-shaped}, if it converges to zero at $\pm\infty$ and for every $n = 0, 1, 2, \ldots$ the $n$th iterated difference $(\Delta^n a_k)$ changes sign exactly $n$ times. A one-sided sequence is bell-shaped if the corresponding two-sided sequence, obtained by padding zeroes for negative indices, is bell-shaped. Although the notion of a bell-shaped sequence seems very natural, apparently it has not yet appeared in mathematical literature.

It is rather straightforward to verify that geometric sequences and, more generally, completely monotone sequences are bell-shaped. In Section~\ref{sec:examples} we provide a few more classes of bell-shaped sequences. The main purpose of this article is to prove the following characterisation theorem, which is a discrete analogue (for one-sided sequences) of the complete description of the class of bell-shaped functions developed in~\cite{js,kwasnicki,ks}.

We say that a function $\ph$ is \emph{increasing-after-rounding} if there is an increasing integer-valued function $\tilde{\ph}$ such that $\tilde{\ph} \le \ph \le \tilde{\ph} + 1$. Note that if $\lfloor \ph \rfloor$ or $\lceil \ph \rceil$ is increasing, then $\ph$ is increasing-after-rounding, but the converse is not quite true: the latter condition is slightly more general.

\begin{theorem}\label{thm:main}
For a nonnegative sequence $(a_k)$, the following are equivalent:
\begin{enumerate}[label={\rm(\alph*)}]
\item\label{thm:main:a} $(a_k)$ is a bell-shaped sequence;
\item\label{thm:main:b} $(a_k)$ is the convolution of a summable Pólya frequency sequence and a completely monotone sequence which converges to zero;
\item\label{thm:main:c} the generating function of $(a_k)$ is given by the formula
\formula[eq:main]{
 \sum_{k = 0}^\infty a_k x^k & = \exp \biggl( b x + c + \int_{-\infty}^\infty \biggl( \frac{1}{s - x} - \frac{s}{1 + s^2} \biggr) \ph(s) ds \biggr)
}
for $x \in (0, 1)$, where $b \in [0, \infty)$, $c \in \R$ and $\ph$ is a nonnegative Borel function on $\R$ such that:
 \begin{itemize}
 \item $\ph$ is decreasing and integer-valued on $(-\infty, 0)$;
 \item $\ph$ is equal to zero on $[0, 1]$;
 \item $\ph$ is increasing-after-rounding on $(1, \infty)$;
 \item $\ph(s) / s^2$ is integrable near $-\infty$ and near $\infty$;
 \item $(1 - \ph(s)) / (1 - s)$ is nonintegrable in a right neighbourhood of $1$.
 \end{itemize}
\end{enumerate}
Furthermore, the right-hand side of~\eqref{eq:main} is the generating function of a bell-shaped sequence whenever the conditions on $b, c, \ph$ listed in item~\ref{thm:main:c} are satisfied.
\end{theorem}

A sample function $\ph$ satisfying the conditions listed in item~\ref{thm:main:c} is shown in Figure~\ref{fig:bs}.

\begin{figure}
\includegraphics[width=0.8\textwidth]{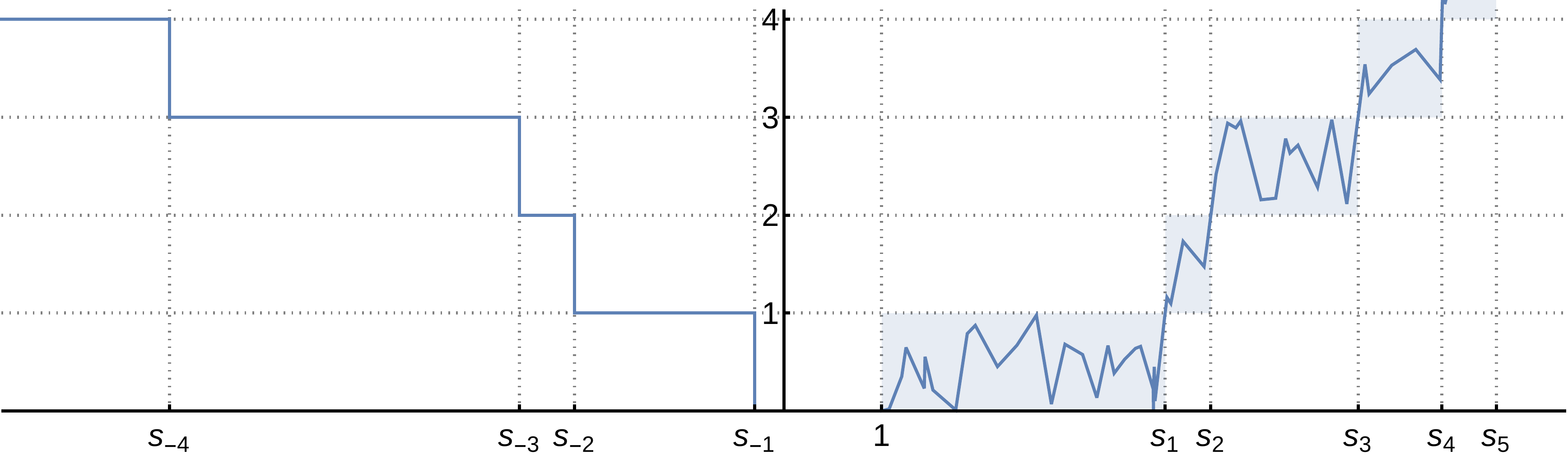}
\caption{A sample function $\ph$ in Theorem~\ref{thm:main}\ref{thm:main:c}. For an appropriate doubly infinite increasing sequence of \emph{points of increase} $s_k \in [-\infty, \infty]$ such that $s_{-1} < 0 < 1 = s_0$, we have $\ph(s) = k$ for $s \in (s_{-k - 1}, s_{-k})$, and $k \le \ph(s) \le k + 1$ for $s \in (s_k, s_{k + 1})$, where $k = 0, 1, 2, \ldots$}
\label{fig:bs}
\end{figure}

\begin{remark}
The above result looks very similar to Theorems~1.1 and~1.3 of~\cite{ks} (see also Theorem~1.1 in~\cite{kwasnicki}), which provide a similar description for bell-shaped functions. For a one-sided integrable function $f$, these results assert that the following three conditions are equivalent: (a)~$f$ is bell-shaped; (b) $f$ is the convolution of a Pólya frequency function and a completely monotone function; (c)~the Laplace transform of $f$ is given by an integral formula similar to~\eqref{eq:main}.

The proof of our Theorem~\ref{thm:main} also follows the approach of~\cite{js,kwasnicki,ks}. Nevertheless, there are essential differences. The most important obstacle that needed to be overcome is related to the factorisation~\eqref{eq:bs:proof:3} of the expression appearing in the discrete variant of Post's inversion formula. In the case of bell-shaped functions, an analogous expression was automatically a product of two Laplace transforms: one of a Pólya frequency function and another one of a completely monotone function (or, for two-sided functions, an AM-CM function). In our case, the first factor in the right-hand side of~\eqref{eq:bs:proof:3} is indeed the generating function of a Pólya frequency sequence. The other one, however, is not the generating function of a completely monotone sequence, not even after passing to the limit as $n \to \infty$. In fact, the two factors need to be treated simultaneously, and additional arguments are needed in order to prove the desired factorisation of the limit.
\end{remark}

\begin{remark}
Our motivation to study bell-shaped sequences, a subclass of unimodal sequences, comes from probability theory, where geometric properties of probability mass functions of discrete random variables plays a certain role; we refer to~\cite{dj} for an account of unimodality, and to~\cite{bdk,dhkr,jw,ll,rao} for a sample of applications in statistics. Thus, we remark that a variant of Theorem~\ref{thm:main} for \emph{summable} bell-shaped sequences holds true, provided that in item~\ref{thm:main:b} the completely monotone sequence is required to be summable rather than converge to zero, and in item~\ref{thm:main:c} the last condition of $\ph$ is replaced by the following stronger one: $\ph(s) / (s - 1)$ is integrable in a right neighbourhood of $1$.

Of course, unimodality, convexity and related properties of sequences find numerous applications outside probability. We refer to the survey~\cite{stanley} for a rather outdated, but still excellent discussion in the context of algebra, combinatorics and geometry.
\end{remark}

The remaining part of the article consists of five sections. Basic definition and auxiliary results are gathered in Section~\ref{sec:pre}. Section~\ref{sec:repr} contains the core of the proof of Theorem~\ref{thm:main}, namely, the proof that condition~\ref{thm:main:a} implies condition~\ref{thm:main:c}. The proof of the other assertions of Theorem~\ref{thm:main} is given in Section~\ref{sec:exp}. In Section~\ref{sec:whale} we discuss briefly a closely related concept of whale-shaped sequences. Finally, we provide a number of examples in Section~\ref{sec:examples}.

%
%

\section{Preliminaries}
\label{sec:pre}

Below we gather definitions and known results, as well as a few auxiliary lemmas, required for the proof of Theorem~\ref{thm:main}.

\subsection{Sequences}

Throughout the article, sequences (or one-sided sequences) are indexed with nonnegative integers $\N = \{0, 1, 2, \ldots\}$, and doubly infinite sequences (or two-sided sequences) are indexed with integers $\Z = \{\ldots, -1, 0, 1, 2, \ldots\}$. To improve clarity, we often use the function notation $a(k)$ for the $k$th term of a sequence instead of the more customary subscript notation $a_k$. We identify every sequence $(a(k) : k \in \N)$ with the corresponding doubly infinite sequence $(\bar{a}(k) : k \in \Z)$, defined by
\formula{
 \bar{a}(k) & = \begin{cases}
  a(k) & \text{for $k \ge 0$,} \\
  0 & \text{for $k < 0$.}
 \end{cases}
}
Whenever this causes no confusion, we do not distinguish between these two sequences, and we use the same symbol $a(k)$ to denote both of them. Additionally, whenever the meaning is clear from the context, we use the symbol $a(k)$ to denote both the $k$th term of a sequence and the entire sequence, and we tend avoid the more formal, but less convenient notation $(a(k) : k \in \N)$ or $(a(k))$.

The \emph{difference operator} $\Delta$ is defined by the formula
\formula{
 \Delta a(k) & = a(k + 1) - a(k) .
}
For $n = 0, 1, 2, \ldots$ we define the \emph{iterated difference operator} $\Delta^n$ inductively: we let $\Delta^0 a(k) = a(k)$ and $\Delta^{n + 1} a(k) = \Delta^n \Delta a(k)$ for $n = 0, 1, 2, \ldots$

The \emph{convolution} of sequences is defined in the usual way:
\formula{
 a * b(k) & = \sum_{j = -\infty}^\infty a(j) b(k - j)
}
whenever the series in the right-hand side converges. Note that the series reduces to a finite sum if both $a(k)$ and $b(k)$ are one-sided sequences.

We have the following summation by parts formula:
\formula{
 \sum_{k = -\infty}^\infty a(k) \Delta b(k) & = -\sum_{k = -\infty}^\infty b(k) \Delta a(k - 1)
}
whenever either of sums converges and the boundary terms go to zero:
\formula{
 \lim_{k \to \pm\infty} a(k) b(k) & = 0 .
}
An $n$-fold application of summation by parts leads to a more general formula
\formula{
 \sum_{k = -\infty}^\infty a(k) \Delta^n b(k) & = (-1)^n \sum_{k = -\infty}^\infty b(k) \Delta^n a(k - n) ,
}
again provided that all boundary terms go to zero:
\formula{
 \lim_{k \to \pm\infty} \Delta^j a(k - j) \Delta^{n - 1 - j} b(k) & = 0
}
for $j = 0, 1, 2, \ldots, n - 1$.

A sequence $a(k)$ is said to \emph{change sign} $N$ times, where $N = 0, 1, 2, \ldots$\,, if there exist indices $k_0 < k_1 < \ldots < k_N$ such that
\formula{
 a(k_{j - 1}) a(k_j) & < 0 & \text{for $j = 1, 2, \ldots, N$,}
}
and $N$ is the largest number with the above property. If such indices exist for every $N$, we say that $a(k)$ changes sign infinitely many times.

A nonnegative sequence $a(k)$ is said to be:
\begin{enumerate}[label=(\alph*)]
\item a \emph{completely monotone sequence} if $(-1)^n \Delta^n a(k) \ge 0$ for $n, k = 0, 1, 2, \ldots$\,;
\item a \emph{Pólya frequency sequence} if the infinite matrix $(a(k - l) : k, l \in \N)$ is \emph{totally positive};
\item a \emph{bell-shaped sequence} if $a(k)$ converges to zero as $k \to \infty$, and for every $n = 0, 1, 2, \ldots$ the doubly infinite sequence $\Delta^n a(k)$ changes sign $n$ times.
\end{enumerate}
To simplify the discussion, we exclude the sequence which is constant zero. All three notions are discussed in more detail later in this section, and the last one is the main subject of the present article.

\subsection{Generating functions and moment sequences}

The \emph{generating function} of a sequence $a(k)$ is given by
\formula{
 F(x) & = \sum_{k = 0}^\infty a(k) x^k
}
whenever the series converges. If $a(k)$ is summable, then the generating function is a holomorphic function in the unit disk in the complex plane, and it extends continuously to the unit circle. When $a(k)$ is merely bounded, then $F$ is still a holomorphic function in the unit disk, but it may fail to extend continuously to the boundary.

The dual notion is the \emph{moment sequence}: if $\mu$ is a finite signed measure on $\R$, then the moment sequence $a(k)$ of $\mu$ is defined as
\formula{
 a(k) & = \int_{\R} s^k \mu(ds)
}
whenever the integral is well-defined. This is the case when, for example, $\mu$ is concentrated on a bounded interval.

The following discrete variant of the classical Post's inversion formula for the Laplace transform plays a key role in our proof. For completeness, we provide a short probabilistic proof.

\begin{theorem}[discrete Post's inversion formula; Theorem III.3 in~\cite{widder}]
\label{thm:post}
Suppose that $F$ is an integrable function on $(0, 1)$. Let
\formula{
 A(k) & = \int_{(0, 1)} x^k F(x) dx ,
}
where $k = 0, 1, 2, \ldots$\,, be the moment sequence of $F$. Suppose furthermore that $F$ is continuous at $x \in (0, 1)$, and
\formula{
 \lim_{n \to \infty} \frac{j_n}{n} & = \frac{x}{1 - x} \, .
}
Then
\formula{
 F(x) & = \lim_{n \to \infty} (n + j_n + 1) \binom{n + j_n}{n} (-1)^n \Delta^n A(j_n) .
}
\end{theorem}

\begin{proof}
We first consider the case when $F$ is bounded on $(0, 1)$. Suppose that $X(1), X(2), \ldots$ is a sequence of i.i.d.\@ random variables uniformly distributed over $(0, 1)$. Let $X(j : k)$ denote the corresponding order statistic, i.e.\@ the $j$th smallest value among $X(1), X(2), \ldots, X(k)$. Suppose that
\formula{
 \lim_{n \to \infty} k_n & = \infty && \text{and} & \lim_{n \to \infty} \frac{j_n}{k_n} & = x .
}
We have
\formula{
 \lim_{n \to \infty} X(j_n : k_n) = x
}
with probability one. (This relatively simple folklore fact follows easily from stronger results: asymptotic normality of order statistics, see~\cite{bahadur}, or the Glivenko--Cantelli theorem.) In particular, by Lebesgue's dominated convergence theorem,
\formula{
 \lim_{n \to \infty} \ex F(X(j_n : k_n)) & = F(x) .
}
On the other hand,
\formula{
 \pr(X(j : k) \in dx) & = (k + 1) \binom{k}{j} x^j (1 - x)^{k - j} dx ,
}
and hence, if $0 < j < k$,
\formula{
 \ex F(X(j : k)) & = \int_0^1 (k + 1) \binom{k}{j} x^j (1 - x)^{k - j} F(x) dx \\
 & = (k + 1) \binom{k}{j} \sum_{i = 0}^{k - j} \binom{k - j}{i} (-1)^i A(j + i) \\
 & = (k + 1) \binom{k}{j} (-1)^{k - j} \Delta^{k - j} A(j) .
}
Choosing $k_n = n + j_n$ and observing that
\formula{
 & \text{if } \lim_{n \to \infty} \frac{j_n}{n} = \frac{x}{1 - x} \, , \text{ then } \lim_{n \to \infty} \frac{j_n}{k_n} = x ,
}
we obtain the desired result.

For a general integrable function $F$, the argument is very similar, but we use Vitali's convergence theorem instead of Lebesgue's dominated convergence theorem. Since we will not need this result here, we only sketch the proof. For a sufficiently small $\delta > 0$, the function $F$ is bounded on $(x - \delta, x + \delta)$. On the other hand, the density functions of $X(j_n : k_n)$ are easily shown to be bounded on $(0, 1) \setminus (x - \delta, x + \delta)$ uniformly with respect to $n = 1, 2, \ldots$\, This implies that the family of random variables $F(X(j_n : k_n))$ is uniformly integrable, and so Vitali's convergence theorem indeed applies to the limit of $\ex F(X(j_n : k_n))$.

\ignore{It remains to show that if $j_n / n$ converges to $\frac{x}{1 - x}$ and $k_n = n + j_n$, then
\formula{
 f_n(s) & = (k_n + 1) \binom{k_n}{j_n} s^j (1 - s)^{k - j_n}
}
is bounded uniformly in $n = 1, 2, \ldots$ and $s \in (0, 1) \setminus (x - \delta, x + \delta)$. This is standard, but for completeness we provide a short argument. Recall that if $x_n = j_n / k_n$, then $x_n$ converges to $x$. The unimodal function $f_n$ (defined on $(0, 1)$) has a maximum at $x_n$, and so for $n$ sufficiently large, $f_n$ is increasing over $(0, x - \frac{\delta}{2}]$ and decreasing over $[x + \frac{\delta}{2}, 1)$. It follows that for $n$ sufficiently large and $s \in (0, x - \delta]$, we have
\formula{
 f_n(s) \le f_n(x - \delta) & \le \frac{2}{\delta} \int_{x - \delta}^{x - \delta/2} f_n(t) dt \\
 & \le \frac{2}{\delta} \int_{(0, 1) \setminus (x - \delta/2, x + \delta/2)} f_n(t) dt \\
 & \le \frac{2}{\delta} \, \pr(|X(j_n : k_n) - x| \ge \tfrac{\delta}{2}) .
}
A similar argument shows that the same estimate holds for $s \in [x + \delta, 1)$. Finally, since $X(j_n : k_n)$ converges to $x$ with probability one, the right-hand side of the above estimate clearly converges to zero as $n \to \infty$, and the desired uniform bound follows.}
\end{proof}

\subsection{Completely monotone sequences}

Recall that a sequence $a(k)$ is \emph{completely monotone} if $(-1)^n \Delta^n a(k) \ge 0$ for $n, k = 0, 1, 2, \ldots$ We assume here that $a(k)$ is not constant zero.

Every geometric sequence $s^k$ with $s \in [0, 1]$ is clearly completely monotone: if $a(k) = s^k$, then $\Delta^n a(k) = (-1)^n (1 - s)^n a(k)$ (here and below we assume that $0^0 = 1$). It follows that for every finite measure $\mu$ on $[0, 1]$ the moment sequence
\formula[eq:hausdorff]{
 a(k) & = \int_{[0, 1]} s^k \mu(ds)
}
is completely monotone. By the famous theorem due to Hausdorff, the converse is true: every completely monotone sequence is the moment sequence of a finite measure on $[0, 1]$, that is, it is given by~\eqref{eq:hausdorff}. We refer to Section~III.4 in~\cite{widder} for further properties of completely monotone sequences.

By Fubini's theorem, the generating function of a completely monotone sequence $a(k)$ given by~\eqref{eq:hausdorff} satisfies
\formula[eq:cm:gen]{
 F(x) & = \sum_{k = 0}^\infty a(k) x^k = \int_{[0, 1]} \sum_{k = 0}^\infty s^k x^k \mu(ds) = \int_{[0, 1]} \frac{1}{1 - s x} \, \mu(ds)
}
when $|x| < 1$. In particular, $F$ extends to a holomorphic function in $\C \setminus [1, \infty)$, and it is a \emph{Pick function} on $(-\infty, 1)$, a notion discussed later in this section; see Section~\ref{sec:cm:pick} for details.

As a direct consequence of Fubini's theorem, a completely monotone sequence $a(k)$ given by~\eqref{eq:hausdorff} converges to zero if and only if $\mu(\{1\}) = 0$, and $a(k)$ is summable if and only if additionally
\formula{
 \int_{[0, 1)} \frac{1}{1 - s} \, \mu(ds) < \infty .
}
In terms of the generating function $F$ given by~\eqref{eq:cm:gen}, $a(k)$ converges to zero if and only if
\formula[eq:cm:zero]{
 \lim_{x \to 1^-} (1 - x) F(x) & = 0 .
}
Indeed: the limit in the left-hand side is equal to $\mu(\{1\})$ by~\eqref{eq:cm:gen} and the dominated convergence theorem. Similarly, $a(k)$ is summable if and only if $F(1)$ is well-defined, or, equivalently, $F$ is bounded on $(0, 1)$.

\subsection{Pólya frequency sequences}
\label{sec:pf}

By definition, a doubly infinite sequence $a(k)$ is a \emph{Pólya frequency sequence} if the infinite matrix $(a(k - l) : k, l \in \N)$ is \emph{totally positive}, that is, if all finite matrices $(a(k_i - l_j) : i, j \in \{1, 2, \ldots, n\})$, where $k_1 < k_2 < \ldots < k_n$ and $l_1 < l_2 < \ldots < l_n$, have nonnegative determinants. Again we assume here that $a(k)$ is not constant zero.

The theory of Pólya frequency sequences is in large part due to Schoenberg. In particular, Edrei proved the conjecture of Schoenberg, which asserts that if $a(k)$ is a (one-sided) Pólya frequency sequence, then the generating function $F$ of $a(k)$ is given by
\formula[eq:pf:gen]{
 F(x) & = \sum_{k = 0}^\infty a(k) x^k = e^{b x + c} \prod_{m = 0}^\infty \frac{1 + q_m x}{1 - p_m x} \, ,
}
where $b \in [0, \infty)$, $c \in \R$, and $p_m$ and $q_m$ are summable nonnegative sequences. Conversely, the right-hand side of the above formula always defines the generating function of a Pólya frequency sequence. This result is given as Theorem~11.5.3 in Karlin's monograph~\cite{karlin} on total positivity.

A Pólya frequency sequence $a(k)$ satisfying~\eqref{eq:pf:gen} is bounded if and only if $p_m \le 1$ for every $m$, and it is summable if and only if $p_m < 1$ for every $m$.

For a summable Pólya frequency sequence $a(k)$, the generating function $F$ again extends to a holomorphic function in $\C \setminus [1, \infty)$, and its complex logarithm (which is continuous and takes a real value at $0$) turns out to be a \emph{Pick function} on $(0, 1)$, as shown later in this section; see Section~\ref{sec:pf:pick} for details. We remark that for a general Pólya frequency sequence $a(k)$, the complex logarithm of the generating function $F$ is a Pick function on $(-(\max_{m \in \N} q_m)^{-1}, (\max_{m \in \N} p_m)^{-1})$, but we will not need this result.

Every summable Pólya frequency sequence $a(k)$ has the \emph{variation diminishing property}: the convolution with $a(k)$ does not increase the number of sign changes; see Theorem~5.1.5 in~\cite{karlin}. Although this will not be needed in this article, we mention that the converse is also true: every summable nonnegative sequence with the variation diminishing property is a Pólya frequency sequence. The last property follows from the results of Chapter~5 in~\cite{karlin} in a similar way as Theorem~5.4.2 therein is proved. See also Chapter~IV in~\cite{hw} for a related discussion.

We remark that if a random variable $X$ is a sum of independent: Poissonian random variable with parameter $b$; series of geometric random variables with parameters $p_m$; and series of Bernoulli random variables with parameters $q_m / (q_m + 1)$, then the sequence $a(k) = \pr(X = k)$ is a Pólya frequency sequence satisfying~\eqref{eq:pf:gen} with an appropriate normalisation constant $c$.

\subsection{Pick functions}

A \emph{Pick function} is a holomorphic map from the (open) upper complex half-plane to its closure. Other names are commonly used for this notion, including: \emph{Nevanlinna function} (not to be confused with the \emph{Nevanlinna class} of functions), \emph{Nevanlinna--Pick function}, \emph{Herglotz function}. By a classical result due to Herglotz, every Pick function $F$ admits the \emph{Stieltjes representation}
\formula[eq:pick]{
 F(x) & = b x + c + \int_{\R} \biggl( \frac{1}{s - x} - \frac{s}{1 + s^2} \biggr) \mu(ds) ,
}
where $b \in [0, \infty)$, $c \in \R$ and $\mu$ is a measure on $\R$ such that $\int_{\R} (1 + s^2)^{-1} \mu(ds) < \infty$. Conversely, the right-hand side of~\eqref{eq:pick} defines a Pick function for every admissible $b$, $c$ and $\mu$. This result is given as Theorem II.I in Donoghue's monograph~\cite{donoghue} on Loewner's theorem.

The measure $\mu$ in~\eqref{eq:pick} can be recovered as the boundary limit of the imaginary part of $F$:
\formula{
 \mu(ds) & = \lim_{t \to 0^+} \frac{1}{\pi} \, \Im F(s + i t) ds ,
}
in the sense of the vague limit of measures; see Lemma~II.1 in~\cite{donoghue}, or formula~(3.10) in Remling's book~\cite{remling} on canonical systems.

Formula~\eqref{eq:pick} can be rephrased in the following convenient way:
\formula[eq:pick:var]{
 F(x) & = c + \int_{\R \cup \{\infty\}} \frac{1 + s x}{s - x} \, \tilde{\mu}(ds) ,
}
where $c \in \R$ is the same constant as in~\eqref{eq:pick}, and $\tilde{\mu}$ is a finite measure on $\R \cup \{\infty\}$, the one-point compactification of $\R$, given by
\formula{
 \tilde{\mu}(ds) & = \frac{1}{1 + s^2} \, \mu(ds) + b \delta_\infty(ds) .
}
Here we understand that for $s = \infty$ the integrand is equal to $(1 + s x) / (s - x) = x$. We refer to formula~(3.9) in~\cite{remling} for further details on the above reformulation.

Suppose that $F$ is a Pick function. If $F$ is not constant zero, then, by the open mapping theorem for holomorphic functions, $F$ has no zeroes in the upper complex half-plane. Therefore, the complex logarithm $\log F$ of $F$ is well defined (in the sense of the principal branch), and since $\Im \log F = \Arg F$ is nonnegative, $\log F$ is again a Pick function. Furthermore, since $\Im \log F = \Arg F$ takes values in $[0, \pi]$, in the Stieltjes representation~\eqref{eq:pick} of $\log F$ we necessarily have $b = 0$ and $\mu$ absolutely continuous, with density function taking values in $[0, 1]$. This leads to the \emph{exponential representation} of Pick functions: every nonzero Pick function $F$ is given by
\formula[eq:pick:exp]{
 F(x) & = \exp \biggl(\gamma + \int_{-\infty}^\infty \biggl( \frac{1}{s - x} - \frac{s}{1 + s^2} \biggr) \ph(s) ds\biggr)
}
for some constant $\gamma \in \R$ and some Borel function $\ph$ on $\R$ with values in $[0, 1]$. Conversely, the right-hand side of~\eqref{eq:pick:exp} defines a nonzero Pick function for all admissible parameters $\gamma$ and $\ph$. Furthermore,
\formula{
 \ph(s) ds & = \lim_{t \to 0^+} \frac{1}{\pi} \, \Arg F(s + i t) ds
}
in the sense of vague convergence of measures, and, in fact, $\ph(s)$ is the pointwise limit of $\pi^{-1} \Arg F(s + i t)$ for almost every $s \in \R$. We refer to formula~(II.6) in~\cite{donoghue}, and to Section~7.2 in~\cite{remling}.

We say that $F$ is a \emph{Pick function on $(\alpha, \beta)$}, where $-\infty \le \alpha < \beta \le \infty$, if $F$ is a Pick function which extends continuously to the interval $(\alpha, \beta)$ on the real axis, and takes real values there. By Schwarz's reflection principle, in this case $F$ extends to a holomorphic function on $\C \setminus ((-\infty, \alpha] \cup [\beta, \infty))$, satisfying $F(\overline{x}) = \overline{F(x)}$.

In terms of the Stieltjes representation~\eqref{eq:pick}, a Pick function $F$ is a Pick function on $(\alpha, \beta)$ if and only if $\mu((\alpha, \beta)) = 0$; see Lemma~II.2 in~\cite{donoghue}. Thus, if $F$ is a Pick function on $(\alpha, \beta)$, then representations~\eqref{eq:pick} and~\eqref{eq:pick:var} are valid for all $x \in \C \setminus ((-\infty, \alpha] \cup [\beta, \infty))$. In particular, Pick functions on $(\alpha, \beta)$ are increasing on $(\alpha, \beta)$.

Similarly, a Pick function $F$ is a Pick function on $(\alpha, \beta)$ such that $F > 0$ on $(\alpha, \beta)$ if and only if $\ph = 0$ almost everywhere on $(\alpha, \beta)$ in the exponential representation~\eqref{eq:pick:exp}.

\subsection{Convergence of Pick functions}

A sequence of Pick functions converges pointwise in the upper complex half-plane if and only if it converges uniformly on compact subsets of the upper complex half-plane, and the limit is necessarily again a Pick function. Furthermore, convergence of a sequence of Pick functions is equivalent to the convergence of the corresponding parameters $c$ (usual convergence of real numbers) and $\tilde{\mu}$ (weak convergence of measures on $\R \cup \{\infty\}$) in the modified Stieltjes representation~\eqref{eq:pick:var}. This is essentially Theorem~7.3(a) in~\cite{remling}, see also Lemma~II.3 in~\cite{donoghue}.

A sequence of Pick functions on $(\alpha, \beta)$ converges pointwise on $(\alpha, \beta)$ if and only if it converges in the sense described above, and the limit is necessarily again a Pick function on $(\alpha, \beta)$. This can be proved as Theorem~7.3(a) in~\cite{remling}, using a modified criterion for compactness of a set of Pick functions: Lemma~II.4 in~\cite{donoghue}. Another approach is to apply the much more general result given in Theorem~7.4(b) in~\cite{remling}.

We remark that Pick functions on $(0, \infty)$ which are nonnegative on $(0, \infty)$ form the class of \emph{complete Bernstein functions}, while the class of functions $F$ such that $-F$ is a Pick function on $(0, \infty)$ which is nonpositive on $(0, \infty)$ is the class of \emph{Stieltjes functions}. Pick functions on $(0, \infty)$ are called \emph{extended complete Bernstein functions} in~\cite{ssv}. We refer to that book for a detailed discussion of these classes of functions, and here we only mention that Stieltjes function have Stieltjes representation
\formula[eq:stieltjes]{
 F(x) & = c + \int_{[0, \infty)} \frac{1}{s + x} \, \mu(ds) ,
}
and that if $F$ is a nonzero Stieltjes function, then $1 / F$ is a Pick function on $(0, \infty)$ which is positive on $(0, \infty)$, and thus
\formula[eq:stieltjes:exp]{
 \frac{1}{F(x)} & = \exp \biggl(\gamma + \int_{-\infty}^0 \biggl( \frac{1}{s - x} - \frac{s}{1 + s^2} \biggr) \ph(s) ds\biggr) ,
}
where $\ph$ is a Borel function with values in $[0, 1]$.

\subsection{Generating functions of completely monotone sequences}
\label{sec:cm:pick}

We argue that, as already remarked above, the generating functions of completely monotone sequences are Pick functions.

Let $a(k)$ be a completely monotone function, and let $F$ be the generating function of $a(k)$, given by~\eqref{eq:cm:gen}. It is straightforward to see that the right-hand side of this formula defines a Pick function on $(-\infty, 1)$ which is nonnegative on $(-\infty, 1)$. It follows that the exponential representation of $F$ takes form
\formula[eq:cm:exp]{
 F(x) & = \exp \biggl(\gamma + \int_1^\infty \biggl( \frac{1}{s - x} - \frac{s}{1 + s^2} \biggr) \ph(s) ds\biggr) ,
}
where $\gamma \in \R$ and $\ph$ is a Borel function on $(1, \infty)$ taking values in $[0, 1]$; see Figure~\ref{fig:pff:cm}. Furthermore,
\formula{
 \ph(s) & = \lim_{t \to 0^+} \frac{1}{\pi} \, \Arg F(s + i t)
}
for almost every $s \in (1, \infty)$. Conversely, it is easy to see that every $\gamma \in \R$ and every Borel function $\ph$ on $(1, \infty)$ with values in $[0, 1]$ correspond in the way described above to a completely monotone sequence.

Observe that
\formula{
 \log \frac{1}{1 - x} & = \frac{\log 2}{2} + \int_1^\infty \biggl( \frac{1}{s - x} - \frac{s}{1 + s^2} \biggr) ds .
}
Thus, with the notation introduced above,
\formula{
 \log \frac{1}{(1 - x) F(x)} & = \frac{\log 2}{2} - \gamma + \int_1^\infty \biggl( \frac{1}{s - x} - \frac{s}{1 + s^2} \biggr) (1 - \ph(s)) ds .
}
By the above identity and~\eqref{eq:cm:zero}, a completely monotone sequence $a(k)$ converges to zero if and only if
\formula{
 \lim_{x \to 1^-} \int_1^\infty \biggl( \frac{1}{s - x} - \frac{s}{1 + s^2} \biggr) (1 - \ph(s)) ds & = \infty .
}
By the monotone convergence theorem, the above condition is equivalent to
\formula{
 \int_1^\infty \biggl( \frac{1}{s - 1} - \frac{s}{1 + s^2} \biggr) (1 - \ph(s)) ds & = \infty ,
}
that is, to nonintegrability of $(1 - \ph(s)) / (s - 1)$ in a right neighbourhood of $1$.

Finally, $a(k)$ is summable if and only if $F(x)$ is bounded on $(0, 1)$, or, equivalently,
\formula{
 \lim_{x \to 1^-} \int_1^\infty \biggl( \frac{1}{s - x} - \frac{s}{1 + s^2} \biggr) \ph(s) ds & < \infty .
}
By the monotone convergence theorem, the above condition is equivalent to 
\formula{
 \int_1^\infty \biggl( \frac{1}{s - 1} - \frac{s}{1 + s^2} \biggr) \ph(s) ds & < \infty ,
}
that is, to integrability of $\ph(s) / (s - 1)$ in a right neighbourhood of $1$.

\begin{figure}
\includegraphics[width=0.8\textwidth]{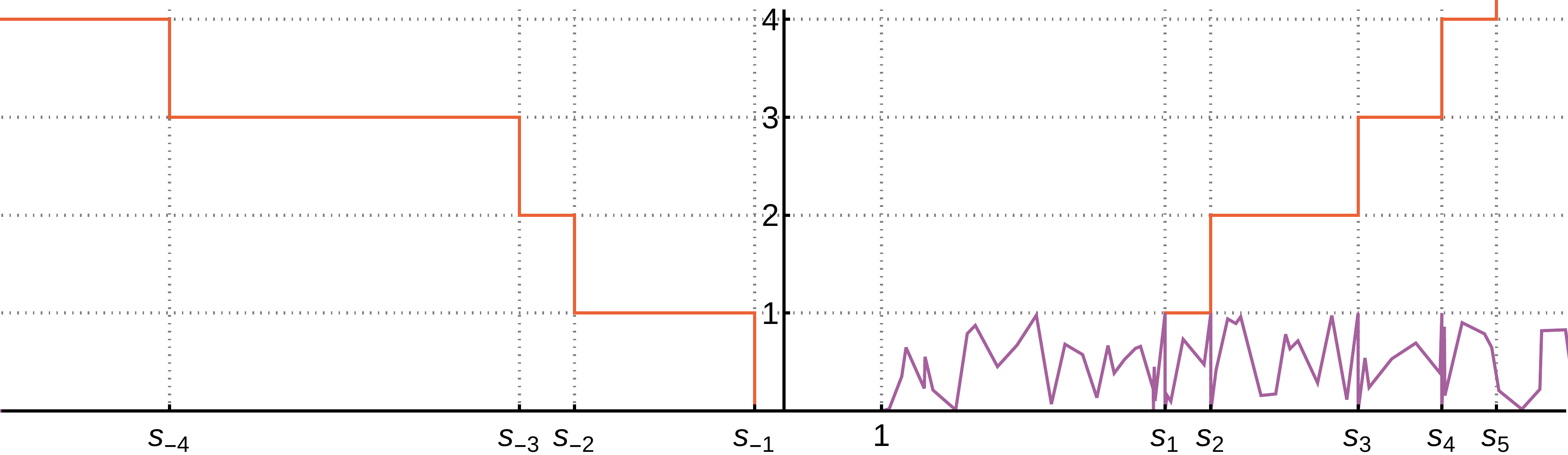}
\caption{A sample function $\ph$ in the representation~\eqref{eq:cm:exp} of the generating function of a completely monotone sequence (purple) and in the representation~\eqref{eq:pf:exp} of the generating function of a Pólya frequency sequence (red). The corresponding sequences are convolution factors of a bell-shaped sequence, which in turn corresponds to the function $\ph$ depicted in Figure~\ref{fig:bs}.}
\label{fig:pff:cm}
\end{figure}

\subsection{Generating functions of Pólya frequency sequences}
\label{sec:pf:pick}

We now show that the generating functions of Pólya frequency sequences are exponentials of Pick functions.

Consider a summable Pólya frequency sequence $a(k)$, and again let $F$ denote the generating function of $a(k)$. By~\eqref{eq:pf:gen}, we have
\formula[eq:pf:gen:log]{
 F(x) & = \exp \biggl( b x + c + \sum_{m = 0}^\infty \log(1 + q_m x) - \sum_{m = 0}^\infty \log(1 - p_m x) \biggr) ,
}
where $\log$ in the right-hand side denotes the principal branch of the complex logarithm. Recall that here $p_m$ and $q_m$ are summable nonnegative sequences, and $p_m < 1$ for every $m$. It is again straightforward to verify that the exponent in the right-hand side of~\eqref{eq:pf:gen:log} defines a Pick function on $(0, 1)$, and since we have
\formula{
 \log(1 + q x) & = \int_{-\infty}^{1/q} \biggl( \frac{1}{s - x} - \frac{1}{s - 1} \biggr) ds , \\
 \log \frac{1}{1 - p x} & = \int_{1/p}^\infty \biggl( \frac{1}{s - x} - \frac{1}{s + 1} \biggr) ds
}
(where we agree that $1/0 = \infty$ and $-1/0 = -\infty$), we find that the Stieltjes representation~\eqref{eq:pick} of $\log F$ reads
\formula*[eq:pf:exp]{
 F(x) & = \exp \biggl( b x + c + \int_{-\infty}^\infty \biggl( \frac{1}{s - x} - \frac{\sign s}{1 + |s|} \biggr) \ph(s) ds \biggr) \\
 & = \exp \biggl( b x + \tilde{c} + \int_{-\infty}^\infty \biggl( \frac{1}{s - x} - \frac{s}{1 + s^2} \biggr) \ph(s) ds \biggr) ,
}
where
\formula[eq:pf:exp:phi]{
 \ph(s) & = \sum_{m = 0}^\infty \ind_{(-\infty, -1/q_m)}(s) + \sum_{m = 0}^\infty \ind_{(1/p_m, \infty)}(s)
}
and
\formula{
 \tilde{c} & = c + \int_{-\infty}^\infty \biggl( \frac{s}{1 + s^2} - \frac{\sign s}{1 + |s|} \biggr) \ph(s) ds ;
}
see Figure~\ref{fig:pff:cm} Note that in the derivation of~\eqref{eq:pf:exp}, we applied Fubini's theorem to exchange the sum and the integral. Another way to justify this step is to observe that adding a series of Pick function is equivalent to adding the corresponding parameters $b$, $c$ and $\mu$ in their Stieltjes representations~\eqref{eq:pick}.

The converse to the above representation is true: every $b \in [0, \infty)$, $c \in \R$ and every Borel function $\ph$ as in~\eqref{eq:pf:exp:phi}, where $p_m$ and $q_m$ are summable nonnegative sequences and $p_m < 1$ for every $m = 0, 1, 2, \ldots$\,, correspond to a summable Pólya frequency sequence.

In the terminology introduced in the following section, $\ph$ given by~\eqref{eq:pf:exp:phi} is stepwise decreasing on $(-\infty, 0)$ and stepwise increasing on $(1, \infty)$. Summability of the sequences $p_m$ and $q_m$ is easily found to be equivalent to the integrability condition
\formula[eq:pf:int]{
 & \text{$\ph(s) / s^2$ is integrable near $-\infty$ and near $\infty$.}
}
on the corresponding function $\ph$ given by~\eqref{eq:pf:exp:phi}, and the inequality $p_m > 1$ for every $m = 0, 1, 2, \ldots$ simply means that $\ph = 0$ in a right neighbourhood of $1$.

Note that formula~\eqref{eq:pf:exp} can be treated as the exponential representation of the generating function $F$, similar to~\eqref{eq:pick:exp}. We stress, however, that typically $F$ fails to be a Pick function, and that the functions $\ph$ defined by~\eqref{eq:pf:exp:phi} are typically not bounded from above by $1$.

\subsection{Auxiliary classes of monotone functions}
\label{sec:step:rect}

We say that a real-valued function $\ph$ is \emph{stepwise increasing} on an interval $(\alpha, \beta)$, where $-\infty \le \alpha < \beta \le \infty$, if, on this interval, $\ph$ is increasing and takes only integer values. \emph{Stepwise decreasing} functions are defined in a similar way; this condition already appeared in Theorem~\ref{thm:main}\ref{thm:main:c}. If $\ph$ is a stepwise increasing function on $(\alpha, \beta)$, then we have $\ph(s) = k$ for $s \in (s_k, s_{k + 1})$, where the \emph{points of increase} $s_k \in [\alpha, \beta]$ form a doubly infinite increasing sequence, which converges to $\alpha$ as $k \to -\infty$, and which converges to $\beta$ as $k \to \infty$. Furthermore, the points of increase are determined uniquely by $\ph$.

We say that a Borel real-valued function $\ph$ is \emph{monotone-after-rounding} on $(\alpha, \beta)$ if for every integer $n$ the function $\ph - n$ changes sign at most once in $(\alpha, \beta)$. More precisely, we say that $\ph$ is \emph{increasing-after-rounding} if $\ph - n$ changes sign from negative to positive for some $n \in \Z$, and that $\ph$ is \emph{decreasing-after-rounding} otherwise; both notions include the case when $\ph - n$ has constant sign for every $n \in \Z$. This definition is easily seen to be equivalent to the one given just before the statement of Theorem~\ref{thm:main}. Indeed: for a function $\ph$ which is monotone-after-rounding on $(\alpha, \beta)$, we have $\ph(s) \in [k, k + 1]$ for $s \in (s_k, s_{k + 1})$, where the \emph{points of increase} $s_k \in [\alpha, \beta]$ again form an appropriately chosen doubly infinite sequence convergent to $\alpha$ as $k \to -\infty$ and to $\beta$ as $k \to \infty$. Note, however, that the points of increase need not be defined uniquely: $s_k$ can be any number between $s_k^- = \inf\{s \in (\alpha, \beta) : \ph(s) \ge k\}$ and $s_k^+ = \sup\{s \in (\alpha, \beta) : \ph(s) \le k\}$. Clearly, $s_k^- \le s_k^+$, but it can happen that $s_k^- < s_k^+$ (and in this case $\ph(s) = k$ for $s \in [s_k^-, s_k^+]$).

Observe that if $\ph_1$ is stepwise increasing and $\ph_2$ is a Borel function taking values in $[0, 1]$, then $\ph_1 + \ph_2$ is increasing-after-rounding on $(\alpha, \beta)$ (with the same points of increase $s_k$). Conversely, if $\ph$ is increasing-after-rounding, then we can define a stepwise increasing function $\ph_1$ (with the same numbers $s_k$) such that $\ph_2 = \ph - \ph_1$ only takes values in $[0, 1]$.

Note that with the above definition, formula~\eqref{eq:pf:exp:phi} can be equivalently phrased as follows: $\ph$ is stepwise decreasing over $(-\infty, 0)$, stepwise increasing over $(1, \infty)$, and equal to zero in an open interval containing $[0, 1]$.

Suppose that $\ph_n$ are Borel functions with values in $[0, 1]$ on an interval $(\alpha, \beta)$, and that the measures $\ph_n(s) ds$ converge vaguely to a measure $\mu$. Then $\mu$ has a density function on $(\alpha, \beta)$ which takes values in $[0, 1]$. Indeed: for every nonnegative continuous $f$ whose support is a compact subset of $(\alpha, \beta)$ we have
\formula{
 0 & \le \int_{-\infty}^\infty f(s) \ph_n(s) ds \le \int_{-\infty}^\infty f(s) ds ,
}
and hence
\formula{
 0 & \le \int_{-\infty}^\infty f(s) \mu(ds) \le \int_{-\infty}^\infty f(s) ds .
}
This implies that on $(\alpha, \beta)$ the measure $\mu$ is absolutely continuous with respect to the Lebesgue measure, with density function taking values in $[0, 1]$, as desired.

The following result asserts that, similarly, if $\ph_n$ are stepwise monotone or monotone-after-rounding on an interval $(\alpha, \beta)$ and the measures $\ph_n(s) ds$ converge vaguely to a measure $\mu$, then $\mu$ has a density function on $(\alpha, \beta)$ which is stepwise monotone or monotone-after-rounding, respectively. The proof of this fact is elementary, but somewhat lengthy.

\begin{lemma}
\label{lem:vague}
Suppose that $\ph_n$ are real-valued functions on $\R$ such that the sequence of signed measures $\ph_n(s) ds$ converges vaguely to a signed measure $\mu$.
\begin{enumerate}[label={\rm (\alph*)}]
\item\label{lem:vague:a} If $\ph_n$ are stepwise increasing on an interval $(\alpha, \beta)$ with points of increase $s_{n, k}$, then $\mu$ has a stepwise increasing density function on $(\alpha, \beta)$. More precisely, for every integer $k$ the sequence $s_{n, k}$ converges as $n \to \infty$ to some limit $s_k$, and $s_k$ is the sequence of points of increase of the density function of $\mu$ on $(\alpha, \beta)$.
\item\label{lem:vague:b} If $\ph_n$ are increasing-after-rounding on an interval $(\alpha, \beta)$, then $\mu$ has an increasing-after-rounding density function on $(\alpha, \beta)$. More precisely, if $s_{n, k}$ denote the points of increase of $\ph_n$ and if $s_k$ denote any partial limit of $s_{n, k}$ as $n \to \infty$, then $s_k$ is a sequence of points of increase of the density function of $\mu$ on $(\alpha, \beta)$.
\end{enumerate}
Similar results hold true for stepwise decreasing functions and functions which are decreasing-after-rounding.
\end{lemma}

\begin{proof}
The proof uses the following property: if each of the functions $\psi_n$ is nonpositive in $(\alpha, s_n)$ and nonnegative in $(s_n, \beta)$, and the sequence of measures $\psi_n(s) ds$ converges vaguely to a measure $\mu$, then for any partial limit $\tilde{s}$ of $s_n$ the measure $\mu$ is nonpositive on $(\alpha, \tilde{s})$ and nonnegative on $(\tilde{s}, \beta)$. This follows directly from the definition of vague convergence: if $f$ is a nonnegative continuous function whose support is a compact subset of $(\alpha, \tilde{s})$, then
\formula{
 \int_{-\infty}^\infty f(s) \psi_n(s) ds & \le 0 \quad \text{infinitely often,}
}
and hence
\formula{
 \int_{-\infty}^\infty f(x) \mu(ds) & \le 0.
}
Thus, $\mu$ is a nonpositive measure on $(\alpha, \tilde{s})$. A similar argument shows that $\mu$ is a nonnegative measure on $(\tilde{s}, \beta)$.

We first prove part~\ref{lem:vague:b}. Suppose that $\ph_n$ are increasing-after-rounding on $(\alpha, \beta)$ and that $\ph_n(s) ds$ converge vaguely to $\mu$. By the above property applied to $\psi_n = \ph_n - k$, for every integer $k$ there is a number $\tilde{s}_k \in [\alpha, \beta]$ such that $\mu(ds) - k \, ds$ is nonpositive on $(\alpha, \tilde{s}_k)$ and nonnegative on $(\tilde{s}_k, \beta)$. Clearly, $\tilde{s}_k$ are increasing, and if $\tilde{s}_i < \tilde{s}_j$, then on $(\tilde{s}_i, \tilde{s}_j)$ we have
\formula{
 i \, ds & \le \mu(ds) \le j \, ds .
}
Thus, $\mu$ has a density function $\ph$ on $(\tilde{\alpha}, \tilde{\beta})$, where
\formula{
 \tilde{\alpha} & = \lim_{k \to -\infty} \tilde{s}_k , \\
 \tilde{\beta} & = \lim_{k \to \infty} \tilde{s}_k ,
}
and $k \le \ph(s) \le k + 1$ for $s \in (\tilde{s}_k, \tilde{s}_{k + 1})$. In other words, $\ph$ is increasing-after-rounding on $(\tilde{\alpha}, \tilde{\beta})$. Finally, suppose that $\tilde{\beta} < \beta$. Since $(\tilde{\beta}, \beta) \sub (\tilde{s}_k, \beta)$ for every integer $k$, we have
\formula{
 \mu(ds) & \ge k \, ds
}
on $(\tilde{\beta}, \beta)$ for every integer $k$, which is absurd. Thus, $\tilde{\beta} = \beta$, and a similar argument shows that $\tilde{\alpha} = \alpha$. The desired result follows.

The proof of part~\ref{lem:vague:a} is very similar. Suppose that $\ph_n$ are stepwise increasing on $(\alpha, \beta)$, with points of increase denoted by $s_{n, k}$, and assume that $\ph_n(s) ds$ converge vaguely to $\mu$. Fix an integer $k$, and choose $\lambda \in (k - 1, k)$. Observe that $\psi_n = \ph_n - \lambda$ changes sign only once, at $s_{n, k}$, and the location of the sign change does not depend on $\lambda$. By the property discussed in the first part of the proof, there exists $\tilde{s}_k$ such that $\mu(ds) \le \lambda \, ds$ on $(\alpha, \tilde{s}_k)$ and $\mu(ds) \ge \lambda \, ds$ on $(\tilde{s}_k, \beta)$ for every $\lambda \in (k - 1, k)$. Thus, we have $\mu(ds) \le (k - 1) ds$ on $(\alpha, \tilde{s}_k)$ and $\mu(ds) \ge k \, ds$ on $(\tilde{s}_k, \beta)$. This shows that $\mu$ has a stepwise increasing density function on an interval $(\tilde{\alpha}, \tilde{\beta})$ defined as in the proof of part~\ref{lem:vague:b}, and the same argument as before shows that in fact $\tilde{\alpha} = \alpha$ and $\tilde{\beta} = \beta$. In order to complete the proof, it remains to observe that the numbers $\tilde{s}_k$ are determined uniquely by $\mu$, and so all partial limits of $s_{n, k}$ as $k \to \infty$ are necessarily equal to $s_k$.
\end{proof}

\subsection{Basic properties of bell-shaped sequences}
\label{sec:bell}

Recall that a sequence $a(k)$ is said to be \emph{bell-shaped} if it is nonnegative, it converges to zero, and, extended to a doubly infinite sequence in such a way that $a(k) = 0$ for $k < 0$, it satisfies the following sign-change condition: for $n = 0, 1, 2, \ldots$ the sequence $\Delta^n a(k)$ changes sign exactly $n$ times. Note that a sequence which is constant zero is not bell-shaped.

Right from the definition it follows that if $a(k)$ is bell-shaped and $n = 0, 1, 2, \ldots$\,, then $(-1)^n \Delta^n a(k) \ge 0$ for $k$ large enough. Using this property with $n$ replaced by $n + 1$, we find that the sequence $(-1)^n \Delta^n a(k)$ is eventually decreasing (that is, decreasing for $k$ large enough).

\begin{lemma}
\label{lem:bell:bound}
If $a(k)$ is a bell-shaped sequence and $n = 0, 1, 2, \ldots$\,, then
\formula{
 \lim_{k \to \pm\infty} k^n \Delta^n a(k) & = 0 .
}
If additionally $a(k)$ is summable, then
\formula{
 \lim_{k \to \pm\infty} k^{n + 1} \Delta^n a(k) & = 0 .
}
\end{lemma}

\begin{proof}
Note that $a(k)$ converges to zero as $k \to \infty$ by definition. Furthermore, if $a(k)$ is additionally summable, then
\formula{
 k a(k) & = \sum_{j = 0}^\infty a(k) \ind_{[1, k]}(j) .
}
Since $a(k) \ind_{[1, k]}(j) \le a(j)$ for every $j$, and $a(k) \ind_{[1, k]}(j)$ converges to $0$ as $k \to \infty$, the dominated convergence theorem implies that $k a(k)$ converges to zero as $k \to \infty$.

By the eventual monotonicity of $\Delta^{n + 1} a(k)$, for $k$ large enough we have
\formula{
 (-1)^n \Delta^n a(\lfloor \tfrac{k}{2} \rfloor) & \ge (-1)^n \Delta^n a(k) - (-1)^n \Delta^n a(\lfloor \tfrac{k}{2} \rfloor) \\
 & = \sum_{j = 1}^{\lceil k/2 \rceil} (-1)^{n + 1} \Delta^{n + 1} a(k - j) \\
 & \ge \lceil \tfrac{k}{2} \rceil (-1)^{n + 1} \Delta^{n + 1} a(k) \ge 0 .
}
Therefore,
\formula{
 |k \Delta^{n + 1} a(k)| & \le 2 |\Delta^n a(\lfloor \tfrac{k}{2} \rfloor)|
}
for $k$ large enough.

The desired results follow now easily by induction.
\end{proof}

By the discrete counterpart of Rolle's theorem, the sequence $\Delta a(k)$ changes sign at least once between each two consecutive changes of sign of $a(k)$. If additionally $a(k)$ converges to zero as $k \to \pm\infty$, then $\Delta a(k)$ additionally changes sign at least once before the first sign change of $a(k)$, and at least once after the last sign change of $a(k)$. By induction, we find that if $a(k)$ converges to zero as $k \to \pm\infty$, then $\Delta^n a(k)$ changes sign at least $n$ times for $n = 0, 1, 2, \ldots$ Consequently, in order to show that such a sequence is bell-shaped, we only need to show that $\Delta^n a(k)$ changes sign at most $n$ times for $n = 0, 1, 2, \ldots$

%
%

\section{Generating functions of bell-shaped sequences}
\label{sec:repr}

In this section we prove the crucial part of our main result: the implication~\ref{thm:main:a}$\implies$\ref{thm:main:c} in Theorem~\ref{thm:main}. For convenience, we state this implication as a separate result: a representation theorem for generating functions of bell-shaped sequences.

\begin{theorem}\label{thm:repr}
If $(a(k))$ is a bell-shaped sequence, then the generating function of $(a(k))$ is given by
\formula[eq:repr]{
 F(x) & = \sum_{k = 0}^\infty a(k) x^k = \exp \biggl( b x + c + \int_{-\infty}^\infty \biggl( \frac{1}{s - x} - \frac{s}{1 + s^2} \biggr) \ph(s) ds \biggr)
}
for $x \in (0, 1)$, where $b \in [0, \infty)$, $c \in \R$ and $\ph$ is a nonnegative Borel function on $\R$ such that
 \begin{itemize}
 \item $\ph$ is stepwise decreasing on $(-\infty, 0)$;
 \item $\ph$ is equal to zero on $(0, 1)$;
 \item $\ph$ is increasing-after-rounding on $(1, \infty)$;
 \item $\ph(s) / s^2$ is integrable near $-\infty$ and near $\infty$;
 \item $(1 - \ph(s)) / (s - 1)$ is nonintegrable in a right neighbourhood of $1$.
 \end{itemize}
\end{theorem}

We first prove the result for summable bell-shaped sequences, and only then we discuss the necessary modifications in the general case.

\begin{proof}[Proof of Theorem~\ref{thm:repr} for summable bell-shaped sequences]
We denote by $\alpha_{n, m}$ the location of the $m$th sign change of $\Delta^n a(k - n)$: we let $\alpha_{n, -1} = -\infty$ and
\formula[eq:bs:proof:alpha]{
 \alpha_{n, m} & = \min\{k > \alpha_{n, m - 1} : (-1)^m \Delta^n a(k - n) > 0\}
}
for $m = 0, 1, 2, \ldots, n - 1$. Note that $0 < \alpha_{n, 0} < \alpha_{n, 1} < \ldots < \alpha_{n, n - 1}$, and thus $\alpha_{n, m} \ge m + 1$.

The argument is broken into ten steps.

\emph{Step 1.} Let $a(k)$ be a summable bell-shaped sequence, and let
\formula{
 F(x) & = \sum_{k = 0}^\infty a(k) x^k
}
be its generating function; here $|x| < 1$. Note that $F$ is continuous on $(0, 1)$, and thanks to summability of $a(k)$, $F$ is additionally bounded on $(0, 1)$. Define the moment sequence
\formula{
 A(k) & = \int_0^1 x^k F(x) dx
}
for $k = 0, 1, 2, \ldots$ By the discrete Post's inversion formula (Theorem~\ref{thm:post}),
\formula[eq:bs:proof:1]{
 F(x) & = \lim_{n \to \infty} (n + j_n + 1) \binom{n + j_n}{n} (-1)^n \Delta^n A(j_n)
}
whenever $x \in (0, 1)$ and
\formula{
 \lim_{n \to \infty} \frac{j_n}{n} & = \frac{x}{1 - x} \, .
}
Below we transform the right-hand side of~\eqref{eq:bs:proof:1}.

\emph{Step 2.} For $k = 0, 1, 2, \ldots$\,, by Fubini's theorem,
\formula{
 A(k) & = \int_0^1 x^k \biggl( \sum_{j = 0}^\infty a(j) x^j \biggr) dx \\
 & = \sum_{j = 0}^\infty \frac{a(j)}{j + k + 1} \, .
}
Therefore,
\formula{
 \Delta^n A(k) & = \sum_{j = 0}^\infty a(j) \Delta_k^n \frac{1}{j + k + 1} \\
 & = \sum_{j = 0}^\infty a(j) \Delta_j^n \frac{1}{j + k + 1} ,
}
where $\Delta_k$ and $\Delta_j$ denote the difference operators acting with respect to variables $k$ and $j$, respectively. Let $P$ be a polynomial of degree at most $n$, and let
\formula{
 Q(j) & = \frac{P(j) - P(-k - 1)}{j + k + 1} \, .
}
Then $Q$ is a polynomial of degree at most $n - 1$, and
\formula{
 \frac{P(-k - 1)}{j + k + 1} & = \frac{P(j)}{j + k + 1} - Q(j) .
}
Additionally, $\Delta^n Q(j) = 0$ for all $j$. Thus,
\formula{
 P(-k - 1) \Delta^n A(k) & = \sum_{j = 0}^\infty a(j) \Delta_j^n \frac{P(-k - 1)}{j + k + 1} \\
 & = \sum_{j = 0}^\infty a(j) \Delta_j^n \frac{P(j)}{j + k + 1} .
}
As usual, we extend $a(k)$ to a two-sided sequence so that $a(k) = 0$ for $k < 0$. The $n$-fold application of summation by parts leads to
\formula*[eq:bs:proof:2]{
 P(-k - 1) \Delta^n A(k) & = \sum_{j = -\infty}^\infty a(j) \Delta_j^n \frac{P(j)}{j + k + 1} \\
 & = \sum_{j = -\infty}^\infty (-1)^n \Delta^n a(j - n) \, \frac{P(j)}{j + k + 1} \, ;
}
here $k = 0, 1, 2, \ldots$\,, and we use the first assertion of Lemma~\ref{lem:bell:bound} to find that for $m = 0, 1, 2, \ldots, n - 1$ the boundary terms
\formula{
 & (-1)^m \Delta^m a(j - m) \Delta_j^{n - m - 1} \frac{P(j)}{j + k + 1}
}
converge to zero as $j \to \infty$.

\emph{Step 3.} By combining the results of the first two steps (formulae~\eqref{eq:bs:proof:1} and~\eqref{eq:bs:proof:2}), we find that for an arbitrary sequence of polynomials $P_n$ of degree at most $n$, we have
\formula*[eq:bs:proof:lim]{
 F(x) & = \lim_{n \to \infty} \frac{n + j_n + 1}{P_n(-j_n - 1)} \binom{n + j_n}{n} \sum_{j = 0}^\infty \Delta^n a(j - n) \, \frac{P_n(j)}{j + j_n + 1} \\
 & = \lim_{n \to \infty} \frac{n + j_n + 1}{n! P_n(-j_n - 1)} \biggl( \prod_{m = 0}^{n - 1} (m + j_n + 1) \biggr) \sum_{j = 0}^\infty \frac{P_n(j) \Delta^n a(j - n)}{j + j_n + 1}
}
for $x \in (0, 1)$, provided that $j_n / n$ converges to $\frac{x}{1 - x}$ as $n \to \infty$. We choose $P_n$ in such a way that $P_n(j) \Delta^n a(j - n) \ge 0$ for all $j = 0, 1, 2, \ldots$ More precisely, we set
\formula[eq:bs:proof:pn]{
 P_n(j) & = \prod_{m = 0}^{n - 1} (\alpha_{n, m} - j) ,
}
where $0 < \alpha_{n, 0} < \alpha_{n, 1} < \ldots < \alpha_{n, n - 1}$ are the locations of sign changes of $\Delta^n a(k - n)$. Thus,
\formula*[eq:bs:proof:3]{
 F(x) & = \lim_{n \to \infty} \frac{n + j_n + 1}{n!} \biggl(\prod_{m = 0}^{n - 1} \frac{m + j_n + 1}{\alpha_{n, m} + j_n + 1} \biggr) \sum_{j = 0}^\infty \frac{P_n(j) \Delta^n a(j - n)}{j + j_n + 1}
}
for $x \in (0, 1)$ and $j_n$ as described above.

\emph{Step 4.} In the right-hand side of~\eqref{eq:bs:proof:3}, the only element that depends on $x$ is the sequence $j_n$: we require that $j_n / n$ converges to $\frac{x}{1 - x}$. We introduce approximations to $F$, by formally replacing $(j_n + 1) / n$ (which also converges to $\frac{x}{1 - x}$) by $\frac{x}{1 - x}$.

More precisely, it will be easier to temporarily fix $n$ and work with variable $y = \frac{n x}{1 - x}$; we will return to the original variable $x$ only in step~9. Thus, we define auxiliary functions $G_n$ by formally replacing $j_n + 1$ by $y$ in the expression under the limit in~\eqref{eq:bs:proof:3}:
\formula[eq:bs:proof:4]{
 G_n(y) & = \frac{n + y}{n!} \biggl( \prod_{m = 0}^{n - 1} \frac{m + y}{\alpha_{n, m} + y} \biggr) \sum_{j = 0}^\infty \frac{P_n(j) \Delta^n a(j - n)}{j + y} \, .
}
With this definition, formula~\eqref{eq:bs:proof:3} simply states that $G_n(j_n + 1)$ converges to $F(x)$ as $n \to \infty$.

\emph{Step 5.} Observe that since $P_n(j) \Delta^n a(j - n) \ge 0$, the series in the right-hand side of~\eqref{eq:bs:proof:4} defines a Stieltjes function
\formula[eq:bs:proof:5]{
 H_n(y) & = \sum_{j = 0}^\infty \frac{P_n(j) \Delta^n a(j - n)}{j + y} \, ;
}
see~\eqref{eq:stieltjes}. The exponential representation~\eqref{eq:stieltjes:exp} of the Pick function $1 / H_n$ reads
\formula[eq:bs:proof:6]{
 H_n(y) & = \exp\biggl( -\delta_n - \int_{-\infty}^0 \biggl( \frac{1}{s - y} - \frac{s}{1 + s^2} \biggr) \psi_n(s) ds \biggr)
}
for some $\delta_n \in \R$ and some Borel function $\psi_n$ with values in $[0, 1]$; here $y \in \C \setminus (-\infty, 0]$. Furthermore, $\psi_n(s)$ is given almost everywhere by the boundary limit
\formula{
 \psi_n(s) & = \lim_{t \to 0^+} \frac{1}{\pi} \, \Arg \frac{1}{H_n(s + i t)} \\
 & = -\lim_{t \to 0^+} \frac{\Arg H_n(s + i t)}{\pi} \, .
}
However, by definition~\eqref{eq:bs:proof:5}, $H_n$ is a meromorphic function, it is real-valued on the real axis, the poles of $H_n$ are located at those numbers $-j$ for which we have $P_n(j) \Delta^n a(j - n) > 0$, $j = 0, 1, 2, \ldots$\,, and $H_n$ is strictly decreasing between every two consecutive poles. This implies that $\psi_n$ only takes values $0$ and $1$, or, more precisely,
\formula{
 \psi_n(s) & = \begin{cases}
  1 & \text{if $H_n(s) < 0$,} \\
  0 & \text{if $H_n(s) \ge 0$.}
 \end{cases}
}
for almost all $s \in \R$. For $m = 0, 1, 2, \ldots$ the function $H_n$ is strictly decreasing on $(-m - 1, -m)$, and hence
\formula[eq:bs:proof:beta]{
 \{ s \in (-m - 1, -m) : H_n(s) < 0 \} & = (-\beta_{n, m}, -m) .
}
for some $\beta_{n, m} \in [m, m + 1]$. It follows that
\formula[eq:bs:proof:7]{
 \psi_n & = \sum_{m = 0}^\infty \ind_{[-\beta_{n, m}, -m)}
}
almost everywhere on $\R$.

\emph{Step 6.} We now derive the exponential representation of the remaining factors in the right-hand side of~\eqref{eq:bs:proof:4}. For simplicity, in this step we simply write $\alpha_m = \alpha_{n, m}$ and $\beta_m = \beta_{n, m}$. For $y \in \C \setminus (-\infty, 0]$ we have
\formula{
 \frac{m + y}{\alpha_m + y} & = \exp\biggl( \int_{-\alpha_m}^{-m} \frac{1}{s - y} \, ds \biggr)
}
and
\formula{
 n + y & = \exp\biggl( \frac{\log(1 + n^2)}{2} + \int_{-\infty}^{-n} \biggl( \frac{1}{s - y} - \frac{s}{1 + s^2} \biggr) ds \biggr) .
}
By combining~\eqref{eq:bs:proof:4}, \eqref{eq:bs:proof:6} and the above two formulae, we find that
\formula[eq:bs:proof:8]{
 G_n(y) & = \exp \biggl( \gamma_n + \int_{-\infty}^0 \biggl( \frac{1}{s - y} - \frac{s}{1 + s^2} \biggr) \ph_n(s) ds \biggr)
}
for $y \in \C \setminus (-\infty, 0]$, where $\gamma_n \in \R$ and
\formula{
 \ph_n & = \ind_{(-\infty, -n)} + \sum_{m = 0}^{n - 1} \ind_{[-\alpha_m, -m)} - \psi_n
}
almost everywhere on $\R$. Using additionally~\eqref{eq:bs:proof:7}, we obtain
\formula{
 \ph_n & = \ind_{(-\infty, -n)} + \sum_{m = 0}^{n - 1} \ind_{[-\alpha_m, -m)} - \sum_{m = 0}^\infty \ind_{[-\beta_m, -m)} \\
 & = \sum_{m = 0}^{n - 1} \bigl(\ind_{[-\alpha_m, -m)} - \ind_{[-\beta_m, -m)}\bigr) + \sum_{m = n}^\infty \bigl(\ind_{[-m - 1, -m)} - \ind_{[-\beta_m, -m)}\bigr)
}
almost everywhere. Recall that $m \le \beta_m \le m + 1 \le \alpha_m$. Thus,
\formula[eq:bs:proof:9]{
 \ph_n & = \sum_{m = 0}^{n - 1} \ind_{[-\alpha_m, -\beta_m)} + \sum_{m = n}^\infty \ind_{[-m - 1, -\beta_m)}
}
almost everywhere on $\R$. Clearly, $\ph_n \ge 0$, and so formula~\eqref{eq:bs:proof:8} asserts that $G_n$ is the exponential of a Pick function on $(0, \infty)$ (with parameter $b$ equal to $\gamma_n$ and measure $\mu$ equal to $\ph_n(s) ds$ in the Stieltjes representation~\eqref{eq:pick}).

\emph{Step 7.} The second term in the right-hand side of~\eqref{eq:bs:proof:9} defines a function which takes values in $\{0, 1\}$ (and hence in $[0, 1]$) on $(-\infty, -n)$, and which is equal to zero on $[-n, 0)$. We claim that the first sum defines a function which is stepwise increasing on $(-\infty, -n)$ and stepwise decreasing on $(-n, 0)$. Since the sum of a stepwise monotone function and a function taking values in $[0, 1]$ is monotone-after-rounding, our claim implies that $\ph_n$ is increasing-after-rounding on $(-\infty, -n)$ and stepwise decreasing on $(-n, 0)$.

We continue to write $\alpha_m = \alpha_{n, m}$ and $\beta_m = \beta_{n, m}$. Clearly, the function $\tilde{\ph}_n$ given by the first term in the right-hand side of~\eqref{eq:bs:proof:9}, that is,
\formula{
 \tilde{\ph}_n & = \sum_{m = 0}^{n - 1} \ind_{[-\alpha_m, -\beta_m)} ,
}
only takes integer values. All upward jumps of $\tilde{\ph}_n$ are located at points $-\alpha_m$, where $m = 0, 1, 2, \ldots, n - 1$, and all downward jumps of $\tilde{\ph}_n$ are located at points $-\beta_m$, where again $m = 0, 1, 2, \ldots, n - 1$. Since $-\alpha_n \le -\beta_n$ and $-n \le -\beta_n$, the function $\tilde{\ph}_n$ is stepwise increasing on $(-\infty, -n)$. Thus, in order to prove our claim, we only need to show that $\tilde{\ph}_n$ is stepwise decreasing on $(-n, 0)$, that is, that every upward jump at $-\alpha_m$ in $(-n, 0)$ is cancelled by some downward jump.

Suppose that $j = \alpha_m < n$. Clearly, $j \ge 1$, and $P_n(j) = \prod_{i = 0}^{n - 1} (\alpha_i - j) = 0$. It follows that $H_n$ does not have a pole at $-j$, and hence $H_n$ is decreasing on $(-j - 1, -j + 1)$. If $H_n(-j) < 0$, then $H_n(s) < 0$ for all $s \in (-j, -j + 1)$, and so
\formula{
 \beta_{j - 1} & = j .
}
Thus, the downward jump of $\ph_n$ at $-\beta_{j - 1} = -j$ cancels the upward jump of $\ph_n$ at $-\alpha_m = -j$, as desired. Similarly, if $H_n(-j) \ge 0$, then $H_n(s) \ge 0$ for all $s \in (-j - 1, -j)$, and therefore
\formula{
 \beta_j & = j .
}
Hence, in this case the downward jump of $\ph_n$ at $-\beta_j = -j$ cancels the upward jump of $\ph_n$ at $-\alpha_m$. Our claim follows.

\emph{Step 8.} Let us denote by $\log G_n$ the exponent in the right-hand side of~\eqref{eq:bs:proof:8}, so that $\log G_n$ is the continuous complex logarithm of the function $G_n$, which is real-valued on $(0, \infty)$. By the result of step~6, $\log G_n$ is a Pick function on $(0, \infty)$, and hence it is increasing and concave on $(0, \infty)$ (this follows immediately from~\eqref{eq:bs:proof:8}). We fix $x \in (0, 1)$, and we define
\formula{
 z & = \frac{x}{1 - x} \, , \\
 w & = \frac{\frac{x}{2}}{1 - \frac{x}{2}} \, , \\
 y_n & = \frac{n x}{1 - x} = n z , \\
 j_n & = \biggl\lfloor \frac{n x}{1 - x} \biggr\rfloor = \lfloor y_n \rfloor = \lfloor n z \rfloor , \\
 i_n & = \biggl\lfloor \frac{\frac{n x}{2}}{1 - \frac{x}{2}} \biggr\rfloor = \lfloor n w \rfloor .
}
Observe that $i_n + 1 < y_n < j_n + 1$ for $n$ large enough. By monotonicity,
\formula{
 \log G_n(y_n) & \le \log G_n(j_n + 1) ,
}
while by concavity,
\formula{
 \log G_n(y_n) & \ge \frac{y_n - (i_n + 1)}{(j_n + 1) - (i_n + 1)} \, \log G_n(j_n + 1) + \frac{(j_n + 1) - y_n}{(j_n + 1) - (i_n + 1)} \, \log G_n(i_n + 1) .
}
By~\eqref{eq:bs:proof:3}, we have
\formula{
 \lim_{n \to \infty} G_n(j_n + 1) & = F(x) , \\
 \lim_{n \to \infty} G_n(i_n + 1) & = F(\tfrac{x}{2}) .
}
It follows that,
\formula{
 \limsup_{n \to \infty} \log G_n(y_n) & \le \lim_{n \to \infty} \log G_n(j_n + 1) = \log F(x) .
}
Furthermore, since $\frac{y_n}{n} = z$, $\frac{j_n + 1}{n} \to z$ and $\frac{i_n + 1}{n} \to w$ as $n \to \infty$, we also have
\formula{
 \liminf_{n \to \infty} \log G_n(y_n) & \ge \lim_{n \to \infty} \biggl(\frac{y_n - (i_n + 1)}{(j_n + 1) - (i_n + 1)} \, \log G_n(j_n + 1) \\
 & \hspace*{10em} + \frac{(j_n + 1) - y_n}{(j_n + 1) - (i_n + 1)} \, \log G_n(i_n + 1) \biggr) \\
 & = \frac{z - w}{z - w} \, \log F(x) + \frac{z - z}{z - w} \, \log F(\tfrac{x}{2}) \\
 & = \log F(x) .
}
We have thus proved that for every $x \in (0, 1)$,
\formula{
 \lim_{n \to \infty} G_n\bigg(\frac{n x}{1 - x}\biggr) & = F(x) .
}

\emph{Step 9.} We return to the original variable $x$: we define $F_n$ by the formula
\formula{
 F_n(x) & = G_n(\tfrac{n x}{1 - x}) .
}
Here $x \in \C \setminus ((-\infty, 0] \cup [1, \infty))$, so that $y = \frac{n x}{1 - x} \in \C \setminus (-\infty, 0]$. By the result of the previous step, the functions $F_n$ converge pointwise to $F$ on $(0, 1)$. We claim that $F_n$ is the exponential of a Pick function on $(0, 1)$, and that the Stieltjes representation of this Pick function is equivalent to
\formula[eq:bs:proof:10]{
 F_n(x) & = \exp \biggl( b_n + \biggl( \int_{-\infty}^0 + \int_1^\infty \biggr) \biggl( \frac{1}{s - x} - \frac{s}{1 + s^2} \biggr) \eta_n(s) ds \biggr)
}
for some $b_n \in \R$, where $\eta_n(s) = \ph_n(\tfrac{n s}{1 - s})$. One way to show this is to observe that both $x \mapsto \frac{n x}{1 - x}$ and the inverse function $y \mapsto \frac{y}{n + y}$ are Pick functions which preserve the real axis, and so composition with the former function is an isomorphism between Pick functions on $(0, \infty)$ and Pick functions on $(0, 1)$. However, we choose a more direct approach.

By~\eqref{eq:bs:proof:8}, for $x \in \C \setminus ((-\infty, 0] \cup [1, \infty))$ we have
\formula{
 F_n(x) & = \exp \biggl( \gamma_n + \int_{-\infty}^0 \biggl( \frac{1 - x}{s (1 - x) - n x} - \frac{s}{1 + s^2} \biggr) \ph_n(s) ds \biggr) .
}
All we need to do is to substitute $s = \frac{n r}{1 - r}$ in the right-hand side, and rearrange the integrand. Let us denote by $\log F$ the exponent in the right-hand side. Thus, we have
\formula{
 \log F_n(x) & = \gamma_n + \biggl(\int_{-\infty}^0 + \int_1^\infty \biggr) \biggl( \frac{(1 - x) (1 - r)}{n r (1 - x) - n x (1 - r)} - \frac{n r (1 - r)}{(1 - r)^2 + n^2 r^2} \biggr) \frac{n \eta_n(r)}{(1 - r)^2} \, dr \\
 & = \gamma_n + \biggl(\int_{-\infty}^0 + \int_1^\infty \biggr) \biggl( \frac{(1 - x) (1 - r)}{n (r - x)} - \frac{n r (1 - r)}{(1 - r)^2 + n^2 r^2} \biggr) \frac{n \eta_n(r)}{(1 - r)^2} \, dr \\
 & = \gamma_n + \biggl(\int_{-\infty}^0 + \int_1^\infty \biggr) \biggl( \frac{1 - x}{(1 - r) (r - x)} - \frac{n^2 r}{(1 - r) ((1 - r)^2 + n^2 r^2)} \biggr) \eta_n(r) dr .
}
Observe that
\formula{
 \frac{1 - x}{(1 - r) (r - x)} & = \frac{1}{r - x} + \frac{1}{1 - r} \, ,
}
and hence, by a straightforward calculation,
\formula{
 & \frac{1 - x}{(1 - r) (r - x)} - \frac{n^2 r}{(1 - r)^3 + n^2 r^2 (1 - r)} \\
 & \qquad = \frac{1}{r - x} + \frac{1 - r - n^2 r}{(1 - r)^2 + n^2 r^2} \\
 & \qquad = \biggl( \frac{1}{r - x} - \frac{r}{1 + r^2} \biggr) + \frac{r}{1 + r^2} + \frac{1 - r - n^2 r}{(1 - r)^2 + n^2 r^2} \, .
}
Therefore,
\formula{
 \log F_n(x) & = c_n + \biggl(\int_{-\infty}^0 + \int_1^\infty \biggr) \biggl( \frac{1}{r - x} - \frac{\sign r}{1 + |r|} \biggr) \eta_n(r) dr
}
for an appropriate $c_n \in \R$, as claimed.

Recall that $\ph_n$ is zero on $(0, \infty)$, $\ph_n$ is increasing-after-rounding on $(-\infty, -n)$ and stepwise decreasing on $(-n, 0)$. Since $\eta_n(s) = \ph_n(\frac{n s}{1 - s})$, we find that $\eta_n$ is zero on $(0, 1)$, $\eta_n$ is increasing-after-rounding on $(1, \infty)$, and $\eta_n$ is stepwise decreasing on $(-\infty, 0)$.

\emph{Step 10.} We are now in position to complete the proof. We already know that the exponent in the right-hand side of~\eqref{eq:bs:proof:10}, which we denote by $\log F_n$, is a Pick function on $(0, 1)$, and that $\log F_n$ converge pointwise on $(0, 1)$ to $\log F$. Thus, $\log F$ is a Pick function on $(0, 1)$ (or, strictly speaking, $\log F$ extends to a holomorphic function on $\C \setminus ((-\infty, 0] \cup [1, \infty))$, which is a Pick function on $(0, 1)$). Additionally, the parameters $c$ and $\tilde{\mu}$ in the modified Stieltjes representation~\eqref{eq:pick:var} of Pick functions $\log F_n$ converge to the corresponding parameters of $\log F$. It follows that
\formula{
 \log F(x) & = c + \int_{\R \cup \{\infty\}} \frac{1 + s x}{s - x} \tilde{\mu}(ds)
}
for $x \in \C \setminus ((-\infty, 0] \cup [1, \infty))$, where $c \in \R$ is the limit of $c_n$, and the measure $\tilde{\mu}(ds)$ on $\R \cup \{\infty\}$ is the weak limit of measures $(1 + s^2)^{-1} \eta_n(s) ds$ as $n \to \infty$. We recall that for $s = \infty$, we understand that $\frac{1 + s x}{s - x} = x$.

We transform the above expression into the usual Stieltjes representation~\eqref{eq:pick}: if $\mu(ds) = (1 + s^2) \tilde{\mu}(ds)$ on $\R$ and $b = \tilde{\mu}(\{\infty\})$, then
\formula{
 \log F(x) & = b x + c + \int_{\R} \biggl( \frac{1}{s - x} - \frac{s}{1 + s^2} \biggr) \mu(ds).
}
for $x \in \C \setminus ((-\infty, 0] \cup [1, \infty))$. Furthermore, $\mu$ is the vague limit of measures $\eta_n(s) ds$ on $\R$. Clearly, $\mu((0, 1)) = 0$. Since functions $\eta_n$ are stepwise decreasing on $(-\infty, 0)$, by Lemma~\ref{lem:vague}\ref{lem:vague:a}, $\mu$ has a density function $\eta$ on $(-\infty, 0)$, and $\eta$ is stepwise decreasing on $(-\infty, 0)$. Similarly, $\eta_n$ are increasing-after-rounding on $(1, \infty)$, and so Lemma~\ref{lem:vague}\ref{lem:vague:b} implies that $\mu$ has a density function $\eta$ on $(1, \infty)$, and $\eta$ is increasing-after-rounding on $(1, \infty)$. Finally, we have $\mu(\{0\}) = \mu(\{1\}) = 0$: perhaps the easiest way to see this is to note that $\eta_n$, and therefore also $\eta$, are in fact stepwise decreasing on $(-\infty, 1)$ and increasing-after-rounding on $(0, \infty)$, and so in particular absolutely continuous on $(-\infty, 1) \cup (0, \infty) = \R$. This completes the proof of~\eqref{eq:repr} (except that the role of $\ph$ is played by the function $\eta$).

The first integrability condition, integrability of $\eta(s) / s^2$ near $-\infty$ and $\infty$, is a consequence of the fact that $\log F$ is a Pick function. For a summable bell-shaped sequence $a(k)$, the other integrability condition is the integrability of $\eta(s) / (1 - s)$ in a right neighbourhood of $1$. This property follows from the monotone convergence theorem by the argument already discussed at the end of Section~\ref{sec:cm:pick}.
\end{proof}

\begin{proof}[Proof of Theorem~\ref{thm:repr} in the general case]
In the proof for summable bell-shaped sequences, summability was only used in the very first step: for a general bell-shaped sequence $a(k)$ the generating function $F(x)$ may fail to be integrable on $(0, 1)$ (namely, when $a(k) / k$ is not summable), and so its moment sequence $A(k)$ is not even well-defined. For this reason, we modify the definition of $A(k)$ to
\formula{
 A(k) & = \int_0^1 (x^k - 1) F(x) dx .
}
The integral in the right-hand side converges, because we have $|(x^k - 1) F(x)| \le k (1 - x) F(x)$, and since $a(k)$ is bounded, the function $(1 - x) F(x)$ is bounded on $(0, 1)$. Furthermore,
\formula{
 -\Delta A(k) & = -\int_0^1 (x^{k + 1} - x^k) F(x) dx = \int_0^1 x^k (1 - x) F(x) dx
}
is the moment sequence of the function $(1 - x) F(x)$, which is bounded and continuous on $(0, 1)$. Thus, we may apply the discrete Post's inversion formula (Theorem~\ref{thm:post}) to find that
\formula[eq:bs:proof:1:alt]{
 (1 - x) F(x) & = \lim_{m \to \infty} (m + i_m + 1) \binom{m + i_m}{m} (-1)^{m + 1} \Delta^{m + 1} A(i_m)
}
whenever $x \in (0, 1)$ and $i_m / m$ converges to $\frac{x}{1 - x}$. But this is in fact equivalent to~\eqref{eq:bs:proof:1} if we substitute $i_m = j_{m + 1}$. Indeed: if $j_n / n$ converges to $\frac{x}{1 - x}$, then $i_m / m$ converges to $\frac{x}{1 - x}$, too, and so~\eqref{eq:bs:proof:1:alt} reads
\formula{
 (1 - x) F(x) & = \lim_{m \to \infty} (m + j_{m + 1} + 1) \binom{m + j_{m + 1}}{m} (-1)^{m + 1} \Delta^{m + 1} A(j_{m + 1}) \\
 & = \lim_{n \to \infty} (n + j_n) \binom{n + j_n - 1}{n - 1} (-1)^n \Delta^n A(j_n) \\
 & = \lim_{n \to \infty} \frac{n}{n + j_n + 1} \times \lim_{n \to \infty} (n + j_n + 1) \binom{n + j_n}{n} (-1)^n \Delta^n A(j_n) \\
 & = (1 - x) \lim_{n \to \infty} (n + j_n + 1) \binom{n + j_n}{n} (-1)^n \Delta^n A(j_n) ,
}
as claimed. The remaining part of the proof is exactly the same as in the summable case, with one exception: the final remark about integrability of $\eta(s) / (1 - s)$ in a right neighbourhood of $1$ needs to be replaced by a similar comment about nonintegrability of $(1 - \eta(s)) / (1 - s)$.
\end{proof}

%
%

\section{Exponential representations}
\label{sec:exp}

The implication~\ref{thm:main:a}$\implies$\ref{thm:main:c} in Theorem~\ref{thm:main} was shown above as Theorem~\ref{thm:repr}. Below we complete the proof of Theorem~\ref{thm:main}. For convenience, the remaining two implications are also phrased as separate results.

First, we combine implication~\ref{thm:main:c}$\implies$\ref{thm:main:b} with the final assertion of Theorem~\ref{thm:main}.

\begin{theorem}\label{thm:fact}
Suppose that $F$ is a function on $(0, 1)$ given by the formula
\formula{
 F(x) & = \exp \biggl( b x + c + \int_{-\infty}^\infty \biggl( \frac{1}{s - x} - \frac{s}{1 + s^2} \biggr) \ph(s) ds \biggr) ,
}
where $b \in [0, \infty)$, $c \in \R$ and $\ph$ is a nonnegative Borel function on $\R$ such that
 \begin{itemize}
 \item $\ph$ is stepwise decreasing on $(-\infty, 0)$;
 \item $\ph$ is equal to zero on $(0, 1)$;
 \item $\ph$ is increasing-after-rounding on $(1, \infty)$;
 \item $\ph(s) / s^2$ is integrable near $-\infty$ and near $\infty$;
 \item $(1 - \ph(s)) / (s - 1)$ is nonintegrable in a right neighbourhood of $1$.
 \end{itemize}
Then $F$ is the generating function of a sequence, which is the convolution of a summable Pólya frequency sequence and a completely monotone sequence which converges to zero.
\end{theorem}

\begin{proof}
As discussed in Section~\ref{sec:step:rect}, every function which is increasing-after-rounding is the sum of a stepwise increasing function and a function with values in $[0, 1]$. Thus, we can write $\ph = \ph_1 + \ph_2$, where:
\begin{itemize}
\item $\ph_1$ only takes nonnegative integer values, it is stepwise decreasing on $(-\infty, 0)$, zero on $(0, 1)$, and stepwise increasing on $(1, \infty)$;
\item $\ph_2$ is zero on $(-\infty, 1)$ and it takes values in $[0, 1]$ on $(1, \infty)$.
\end{itemize}
Furthermore, by the integrability conditions imposed on $\ph$, we may assume that:
\begin{itemize}
\item the function $\ph_1(s) / s^2$ is integrable near $-\infty$ and near $\infty$, and $\ph_1 = 0$ in a right neighbourhood of $1$;
\item $(1 - \ph_2(s)) / (s - 1)$ is nonintegrable in a right neighbourhood of $1$.
\end{itemize}
By the arguments described in Section~\ref{sec:pf:pick}, there exists a summable Pólya frequency sequence $b(k)$ with generating function
\formula{
 G(x) = \sum_{k = 0}^\infty b(k) & = \exp \biggl( b x + c + \int_{-\infty}^\infty \biggl( \frac{1}{s - x} - \frac{s}{1 + s^2} \biggr) \ph_1(s) ds \biggr) ;
}
see~\eqref{eq:pf:exp}. Similarly, by the results of Section~\ref{sec:cm:pick}, there is a completely monotone sequence $c(k)$ which converges to zero, with generating function
\formula{
 H(x) = \sum_{k = 0}^\infty c(k) & = \exp \biggl( \int_{-\infty}^\infty \biggl( \frac{1}{s - x} - \frac{s}{1 + s^2} \biggr) \ph_2(s) ds \biggr) ;
}
see~\eqref{eq:cm:exp}. It follows that $F(x) = G(x) H(x)$, and consequently $F$ is the generating function of the convolution of the Pólya frequency sequence $b(k)$ and the completely monotone sequence $c(k)$.
\end{proof}

It remains to show the implication~\ref{thm:main:b}$\implies$\ref{thm:main:a} of Theorem~\ref{thm:main}.

\begin{theorem}\label{thm:var:dim}
Suppose that $a(k)$ is the convolution of a summable Pólya frequency sequence and a completely monotone sequence which converges to zero. Then $a(k)$ is a bell-shaped sequence.
\end{theorem}

\begin{proof}
Suppose that $a(k)$ is the convolution of a summable Pólya frequency sequence $b(k)$ and a completely monotone sequence $c(k)$ which converges to zero. By Fatou's lemma, $a(k)$ converges to zero as $k \to \infty$. In order to prove that $a(k)$ is bell-shaped, we only need to show that for $n = 0, 1, 2, \ldots$\,, the doubly infinite sequence $\Delta^n a(k)$ changes sign at most $n$ times.

We first observe that $c(k)$ is bell-shaped: the sequence $\Delta^n c(k)$ has constant sign for $k \ge 0$, and it is zero for $k < -n$, and thus it changes sign at most $n$ times. By the variation diminishing property of summable Pólya frequency sequences (see Section~\ref{sec:pf}), for every $n = 0, 1, 2, \ldots$\,, the convolution of $\Delta^n c(k)$ and $b(k)$ has at most $n$ sign changes. It remains to observe that this convolution is precisely $\Delta^n a(k)$, and so $a(k)$ is indeed bell-shaped.
\end{proof}

%
%

\section{Whale-shaped sequences}
\label{sec:whale}

In Section~1.4 of~\cite{ks}, the authors introduce the notion of a \emph{whale-shaped function}, which is an intermediate concept between complete monotonicity and bell-shape. A smooth positive-valued function $f$ on $(0, \infty)$ is said to be whale-shaped of order $d \in \{0, 1, 2, \ldots, \infty\}$ if $f$ converges to $0$ at infinity, $f^{(n)}$ changes sign $\min\{n, d\}$ times on $(0, \infty)$ for $n = 0, 1, 2, \ldots$\,, and additionally $\lim_{x \to 0^+} f^{(n)}(x) = 0$ for $n = 0, 1, 2, \ldots, d - 1$. Note that whale-shaped functions of order $0$ are precisely completely monotone functions on $(0, \infty)$, while whale-shaped functions of infinite order are precisely strictly positive bell-shaped functions on $(0, \infty)$.

We define the corresponding notion of \emph{whale-shaped sequences} in the following way. A (one-sided) strictly positive sequence $a(k)$ is whale-shaped of order $d \in \{0, 1, 2, \ldots, \infty\}$ if $a(k)$ converges to zero, and the sequence $\Delta^n a(k)$, restricted to $k \ge -d$, changes sign $\min\{n, d\}$ times for $n = 0, 1, 2, \ldots$\, Here, as usual, we extend the definiton of $a(k)$ to integer $k$ by setting $a(k) = 0$ for $k < 0$.

Note that for a finite order $d$, one can phrase the main part of the above definition in the following equivalent way: after padding the sequence $a(k)$ with $d$ zeroes on the left, the $n$th iterated difference of this one-sided sequence changes sign $\min\{n, d\}$ times for $n = 0, 1, 2, \ldots$\, In this reformulation the sign-change condition is a direct analogue of the corresponding condition for whale-shaped functions, while padding with $d$ zeroes on the left is the discrete counterpart of the boundary condition $\lim_{x \to 0^+} f^{(n)}(x) = 0$ for $n = 0, 1, \ldots, d - 1$.

Note that whale-shaped sequences of order $0$ are precisely completely monotone sequences, while whale-shaped sequences of infinite order coincide with strictly positive bell-shaped sequences. In contrast to the case of whale-shaped functions, it is easy to see that the class of whale-shaped functions of order $d$ increases with~$d$.

Below we prove the following result, which is an analogue of the corresponding characterisation of whale-shaped functions given in Theorem~1.13 in~\cite{ks}.

\begin{theorem}
\label{thm:whale}
For a nonnegative sequence $a(k)$ and $d \in \{0, 1, 2, \ldots\}$, the following are equivalent:
\begin{enumerate}[label={\rm(\alph*)}]
\item\label{thm:whale:a} the sequence $a(k)$ is whale-shaped of order $d$;
\item\label{thm:whale:b} $a(k)$ is the convolution of a completely monotone sequence which converges to zero and no more than $d$ geometric sequences with quotient in $(0, 1)$;
\item\label{thm:whale:c} the generating function of $a(k)$ is given by the formula
\formula[eq:whale]{
 \sum_{k = 0}^\infty a(k) x^k & = \exp \biggl( c + \int_1^\infty \biggl( \frac{1}{s - x} - \frac{s}{1 + s^2} \biggr) \ph(s) ds \biggr)
}
for $x \in (0, 1)$, where $c \in \R$ and $\ph$ is a nonnegative Borel function on $(1, \infty)$ such that:
 \begin{itemize}
 \item $\ph$ is increasing-after-rounding and bounded from above by $d + 1$;
 \item $(1 - \ph(s)) / (1 - s)$ is nonintegrable in a right neighbourhood of $1$.
 \end{itemize}
\end{enumerate}
Furthermore, the right-hand side of~\eqref{eq:whale} is the generating function of a bell-shaped sequence whenever the conditions on $c, \ph$ listed in item~\ref{thm:whale:c} are satisfied.
\end{theorem}

The argument is a variant of the proof of Theorem~\ref{thm:main}, and we only describe the necessary modifications. As it was the case with Theorem~\ref{thm:main}, we divide the proof into three parts.

The proof Theorem~\ref{thm:fact} carries over with only minor change: we additionally know that the function $\ph_1$ is zero on $(-\infty, 0)$ and it is bounded from above by $d$ on $(1, \infty)$. Thus, condition~\ref{thm:whale:c} implies condition~\ref{thm:whale:b}, and additionally the final claim of Theorem~\ref{thm:fact} holds true.

In order to see that condition~\ref{thm:whale:b} implies condition~\ref{thm:whale:a}, we essentially follow the proof of Theorem~\ref{thm:var:dim}, but we need the following well-known lemma.

\begin{lemma}
\label{lem:geometric}
If $b(k)$ is a geometric sequence with nonnegative quotient and $c(k)$ is a sequence such that $c(k)$ and $b * c(k)$ change sign $n$ times, at positions $\alpha_1, \alpha_2, \ldots, \alpha_n$ and $\beta_1, \beta_2, \ldots, \beta_n$ (arranged in an increasing order), respectively, then $\alpha_1 \le \beta_1 < \alpha_2 \le \beta_2 < \ldots < \alpha_n \le \beta_n$. Additionally, for an arbitrary bounded sequence $c(k)$, the sequence $b * c(k)$ changes sign at most as many times as the sequence $c(k)$ does.
\end{lemma}

\begin{proof}
Suppose that $b(k) = q^k$ with $q \ge 0$. If $q = 0$ we have $\beta_j = \alpha_j$ and there is nothing to prove. If $q > 0$, then
\formula{
 q^{-k} (b * c)(k) & = q^{-k} \sum_{j = 0}^k q^{k - j} c(j) = \sum_{j = 0}^k q^{-j} c(j) .
}
For the proof of the second part of the lemma, it remains to observe that the sequence $q^{-j} c(j)$ must change sign between every two consecutive sign changes of the sequence of its cumulative sums. The first assertion of the lemma follows by the same argument: if both sequences change sign equally many times, then the locations of sign changes of these sequences necessarily alternate exactly as in the statement of the lemma.
\end{proof}

\begin{proof}[Proof of the implication \ref{thm:whale:b}$\implies$\ref{thm:whale:a} in Theorem~\ref{thm:whale}]
Suppose that $c(k)$ is a completely monotone sequence, and $b_1(k), b_2(k), \ldots, b_d(k)$ are geometric sequences with quotient in $[0, 1)$; the quotient equal to $0$ corresponds to the sequence $(1, 0, 0, \ldots)$, the neutral element with respect to convolution. We need to prove that the sequence $a(k) = b_1 * b_2 * \ldots * b_d * c(k)$ is whale-shaped of order $d$.

Fix $n = 0, 1, 2, \ldots$\, Observe that
\formula[eq:whale:proof:1]{
 \Delta^n a(k - n) & = b_1 * b_2 * \ldots * b_d * \Delta^n c(k - n) .
}
The one-sided sequence $a(k - n)$ stars with $n$ zeroes. Thus, by the discrete analogue of Rolle's theorem, its $n$th iterated difference necessarily changes sign at least $n$ times. Similarly, the one-sided sequence $\Delta^n c(k - n)$ changes sign at least $n$ times, and since $c(k)$ is a completely monotone sequence, there are exactly $n$ sign changes, at positions $k = 1, 2, \ldots, n$. On the other hand, using~\eqref{eq:whale:proof:1}, Lemma~\ref{lem:geometric} and induction with respect to $d$, we find that the sequence $\Delta^n a(k - n)$ changes sign at most $n$ times, and hence exactly $n$ times: at $k = 1, 2, \ldots, n - d$, and additionally $d$ times for $k > n - d$. This, however, proves that $a(n)$ is whale-shaped of order~$d$.
\end{proof}

Finally, we prove that condition~\ref{thm:whale:a} implies condition~\ref{thm:whale:c}, by following closely the proof of Theorem~\ref{thm:repr}.

\begin{proof}[Proof of the implication \ref{thm:whale:a}$\implies$\ref{thm:whale:c} in Theorem~\ref{thm:whale}]
We consider a whale-shaped sequence $a(k)$ of degree $d$ and $n > d$, and we use the notation introduced in the proof of Theorem~\ref{thm:repr}. By definition, the sequence $\Delta^n a(k - n)$ changes sign $d$ times for indices $k > n - d$, so that the remaining $n - d$ sign changes are necessarily located at $k = 1, 2, \ldots, n - d$. Thus, the locations $\alpha_{n, m}$ of sign changes of that sequence, defined in~\eqref{eq:bs:proof:alpha}, satisfy
\formula{
 \alpha_{n, m} & = m + 1 && \text{for $m = 0, 1, 2, \ldots, n - d - 1$.}
}
It follows that the definition~\eqref{eq:bs:proof:pn} of the polynomial $P_n$ reads
\formula{
 P_n(j) & = \prod_{m = 0}^{n - 1} (\alpha_{n, m} - j) \\
 & = (1 - j) (2 - j) \ldots (n - d - j) \prod_{m = n - d}^{n - 1} (\alpha_{n, m} - j) .
}
In the remaining part of the argument, however, we replace $\alpha_{n, m}$ and $P_n$ with
\formula{
 \tilde{\alpha}_{n, m} & =
 \begin{cases}
  m & \text{for $m = 0, 1, 2, \ldots, n - d - 1$,} \\
  \alpha_{n, m} & \text{for $m = n - d, n - d + 1, \ldots, n - 1$,}
 \end{cases}
}
and
\formula{
 \tilde{P}_n(j) & = \prod_{m = 0}^{n - 1} (\tilde{\alpha}_{n, m} - j) \\
 & = (-j) (1 - j) (2 - j) \ldots (n - d - 1 - j) \prod_{m = n - d}^{n - 1} (\alpha_{n, m} - j) .
}
Note that the expression under the limit in~\eqref{eq:bs:proof:lim} does not depend on the choice of the polynomial $P_n$, and hence the above modification does not affect neither the function $G_n$ defined in~\eqref{eq:bs:proof:4}, nor the corresponding function $\ph_n$ determined by~\eqref{eq:bs:proof:8}.

After these modifications, however, the auxiliary function $H_n$ defined in~\eqref{eq:bs:proof:5} takes form
\formula{
 \tilde{H}_n(y) & = \sum_{j = 0}^\infty \frac{\tilde{P}_n(j) \Delta^n a(j - n)}{j + y} \\
 & = \sum_{j = n - d}^\infty \frac{\tilde{P}_n(j) \Delta^n a(j - n)}{j + y} \, .
}
We define the numbers $\tilde{\beta}_{n, m}$ as in~\eqref{eq:bs:proof:beta}:
\formula{
 \{ s \in (-m - 1, -m) : \tilde{H}_n(s) < 0 \} & = (-\tilde{\beta}_{n, m}, -m) .
}
Observe that $H_n$ is positive on $(-n + d, \infty)$. It follows that $\tilde{\beta}_{n, m} = m$ for $m = 0, 1, \ldots, n - d - 1$. Thus, the analogue of~\eqref{eq:bs:proof:9} reads
\formula{
 \ph_n & = \sum_{m = 0}^{n - 1} \ind_{[-\tilde{\alpha}_m, -\tilde{\beta}_m)} + \sum_{m = n}^\infty \ind_{[-m - 1, -\tilde{\beta}_m)} \\
 & = \sum_{m = n - d}^{n - 1} \ind_{[-\tilde{\alpha}_m, -\tilde{\beta}_m)} + \sum_{m = n}^\infty \ind_{[-m - 1, -\tilde{\beta}_m)} .
}
It follows that $\ph_n = 0$ on $(-n + d, 0)$ and $\ph_n \le d + 1$ everywhere. Consequently, the function $\eta_n$ introduced in step~9 of the proof of Theorem~\ref{thm:repr} is equal to zero on $(1 - \tfrac{n}{d}, 0)$ and it is bounded from above by $d + 1$ everywhere. After passing to the limit as $n \to \infty$, we find that the parameter $b$ is necessarily equal to zero, the function $\eta$ is equal to zero on $(-\infty, 0)$, and it is bounded from above by $d + 1$ on $(1, \infty)$. This completes the proof of~\eqref{eq:whale} (again with the role of $\ph$ played by the function $\eta$).
\end{proof}

%
%

\section{Examples}
\label{sec:examples}

Of course, all Pólya frequency sequences are bell-shaped. Thus, in particular, the probability mass functions of geometric distributions, Poisson distributions, binomial distributions, as well as their convolutions, are all bell-shaped.

Similarly, all completely monotone sequences are bell-shaped.

It is straightforward to check that a uniform distribution over $\{0, 1, 2, \ldots, n - 1\}$ is bell-shaped if and only if $n = 1$ or $n = 2$. A finitely supported distribution is bell-shaped if and only if it is the convolution of Bernoulli distributions.

By a rather straightforward calculation, probability mass functions $a(k)$ of negative binomial distributions are bell-shaped. Indeed: the generating function of such a sequence $a(k)$ is given by
\formula{
 F(x) = \sum_{k = 0}^\infty a(k) x^k & = \biggl(\frac{1 - p}{1 - p x}\biggr)^\lambda \\
 & = \exp \biggl(c + \lambda \int_{1/p}^\infty \biggl(\frac{1}{s - x} - \frac{s}{1 + s^2} \biggr) ds \biggr)
}
for some $p \in (0, 1)$ and $\lambda \in (0, \infty)$, and an appropriate constant $c$, namely, $c = \lambda \log(1 - p) - \lambda \log(1 + p^2)$. Thus, $F(x)$ is given by~\eqref{eq:main} with $b = 0$ and $\ph(s) = \lambda \ind_{(1/p, \infty)}(s)$, and so $a(k)$ is bell-shaped. Alternatively, one can observe that $a(k)$ is the convolution of $\lfloor \lambda \rfloor$ geometric sequences with parameter $p$, and a completely monotone sequence: the probability mass function of a negative binomial distribution with parameters $p$ and $\lambda - \lfloor \lambda \rfloor$.

A slightly less obvious examples of bell-shaped sequences are given by probability mass functions $a(k)$ of \emph{discrete stable distributions}. Such sequences are characterised by their generating functions
\formula{
 F(x) = \sum_{k = 0}^\infty a(k) x^k & = \exp(-\lambda (1 - x)^\nu) ,
}
where $\nu \in (0, 1]$ is a parameter which corresponds to the index of stability and $\lambda > 0$ is the shape parameter. When $\nu = 1$, we recover the usual Poisson distribution. For $\nu \in (0, 1)$, we have
\formula{
 F(x) & = \exp \biggl(c + \frac{\lambda \sin(\nu \pi)}{\pi} \int_1^\infty \biggl(\frac{1}{s - x} - \frac{s}{1 + s^2} \biggr) (s - 1)^\nu ds \biggr)
}
for an appropriate constant $c$, and hence $F(x)$ is given by~\eqref{eq:main} with $b = 0$ and $\ph(s) = \frac{\lambda}{\pi} \sin(\nu \pi) (s - 1)^\nu \ind_{(1, \infty)}(s)$. Thus, in either case $a(k)$ is indeed bell-shaped. For a detailed discussion of discrete stable distributions, we refer to the original article~\cite{sv} by Steutel and van Harn, where this class of discrete distributions was introduced.

%
%

%
%

\end{document}